\documentclass[english]{article}
\usepackage{amsmath}
\usepackage{graphicx}
\usepackage{amssymb}
\usepackage{amsthm}
\usepackage{dsfont}

\usepackage[vlined,tworuled]{algorithm2e}
\SetKwInOut{Input}{input}
\SetKwInOut{Output}{output}

\SetKwFunction{FromInput}{FromInput}
\SetKwFunction{ToOutput}{ToOutput}
\SetKwFunction{Normalize}{Normalize}
\SetKwFunction{RosenstiehlTarjanModified}{Rosenstiehl-Tarjan-Modified}
\SetKwFunction{GetRTwinstackTransitionList}{Get-R-Twinstack-Transition-List}
\SetKwFunction{GetKPrime}{Get-KPrime}
\SetKwFunction{RelativisticHistories}{Relativistic-Histories}
\SetKwFunction{RelativisticSortableCount}{Relativistic-Sortable-Count}

\theoremstyle{definition}
\newtheorem{thm}{Theorem}[section]
\newtheorem{lem}[thm]{Lemma}

\newtheorem*{sclm}{Subclaim}
\newtheorem*{example}{Example}
\newtheorem*{strat1}{Strategy 1}
\newtheorem*{strat2}{Strategy 2}

\providecommand{\tabularnewline}{\\}

\usepackage{cancel}
\usepackage{tikz}
\usepackage{pgf}

\makeatother

\usepackage{babel}

\begin{document}
\title{Methods of computing deque sortable permutations given complete and
incomplete information}
\author{Dan Denton\thanks{Advised by Peter Doyle and Scot Drysdale}}
\date{Version 1.04 dated 3 June 2012 (with additional figures dated 6 August 2012)}
\maketitle
\let\thefootnote\relax\footnotetext{Dartmouth Computer Science Technical Report TR2012-719}

\abstract The problem of determining which permutations can be sorted using 
certain switchyard networks is a venerable problem in computer science
dating back to Knuth in 1968.  In this work, we are interested in permutations which are
sortable on a double-ended queue (called a deque), or on two parallel stacks.  
In 1982, Rosenstiehl and Tarjan
presented an $O\left(n\right)$  algorithm for testing whether a given
permutation was sortable on parallel stacks.  In the same paper, they 
also presented a modification giving $O\left(n\right)$  test for sortability 
on a deque.  We demonstrate a slight error in the
version of their algorithm for testing deque sortability, and present 
a fix for this problem.

The general enumeration problem for both of these classes
of permutations remains unsolved.  What is known is that the growth rate of 
both classes is approximately $\Theta\left(8^n\right)$,
so computing the number of sortable permutations of length $n$, even for small
values of $n$, is difficult to do using any method that must evaluate
each sortable permutation individually. 
As far as we know, the number
of deque sortable permutations was known only up to $n=14$.  This was computed using 
algorithms which effectively generate all sortable permutations.   By using the symmetries 
inherent in the execution of Tarjan's algorithm,
we have developed a new dynamic programming algorithm which can
count the number of sortable permutations in both classes in $O\left(n^5 2^{n}\right)$ time, allowing 
the calculation of the number of deque and parallel stack sortable permutation
for much higher values of $n$ than was previously possible.

Finally, we have examined the problem of trying to sort a
permutation on a deque when the input elements are only revealed 
at the time when they are pushed to the deque.  (Instead of having 
an omniscient view of the input permutation, this corresponds to
encoding the input permutation as a deck of cards which must be
drawn and pushed onto the deque without looking at the 
remaining cards in the deck.)  We show that there are some sortable
permutations which cannot necessarily be sorted correctly on a
deque using only this imperfect information.
{}

\section{Introduction}
In 1968, Donald Knuth first posed the question ``which permutations
can be sorted on certain simple data structures such as stacks and
queues''? \cite{knuth} More generally, these simple data
structure are instances of switchyard networks which take their name
by analogy to railroad switchyards. Switchyard networks consist of
sets of two-way railroad tracks which serve as linear storage elements,
along with one-way railroad track serving as operations for moving
the end element from one storage element to another. Here is a sampling
of well known switchyard networks.

\begin{figure}[ht]
\center{\includegraphics[scale=0.6]{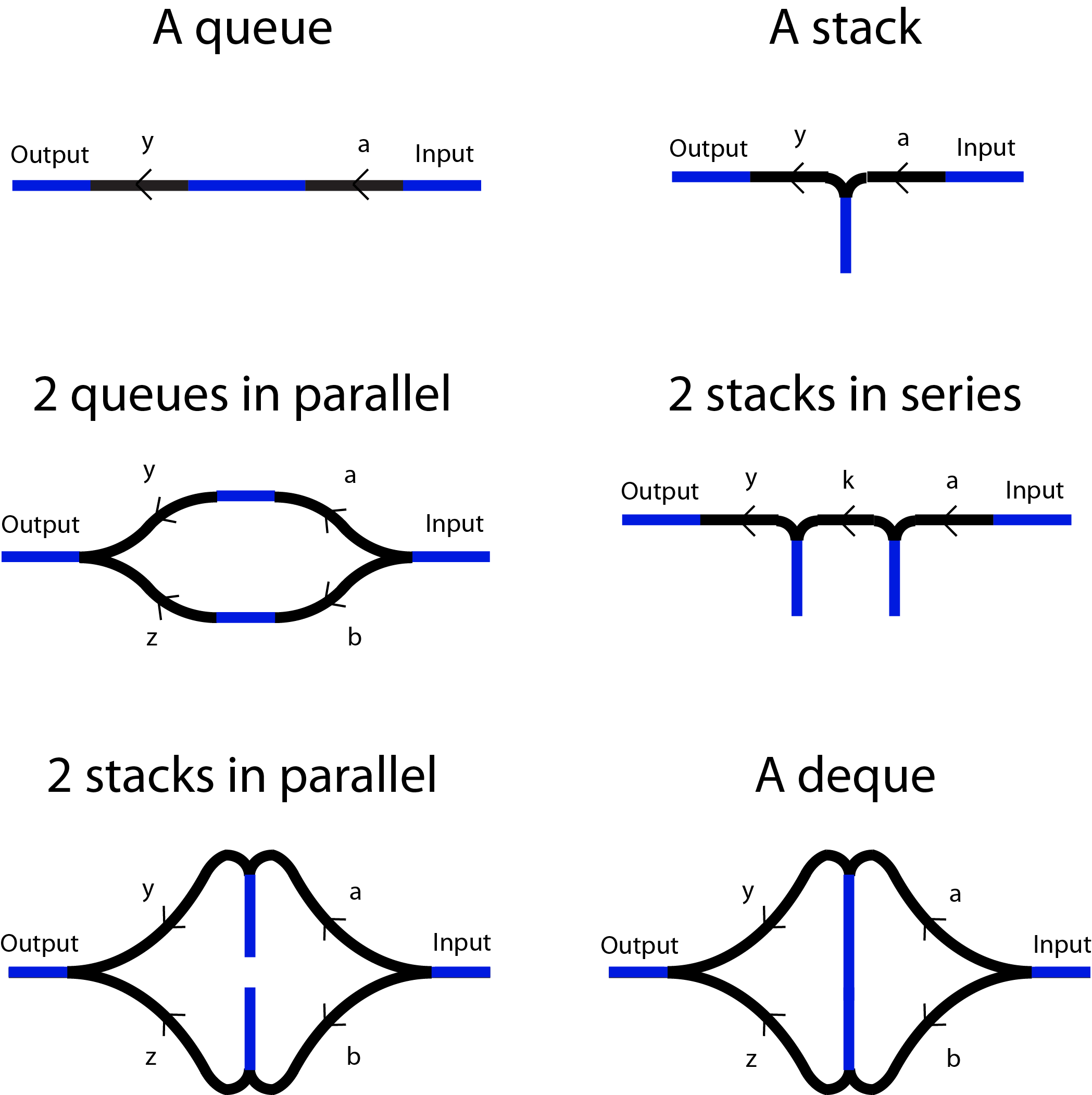}}
\caption{Some small switchyards}
\end{figure}

In the sorting problem posed by Knuth, a permutation $\pi$ initially
sits in the input section of the switchyard. Another section of two-way
track is labeled as the output. The problem, then, is to determine
whether or not the elements of $\pi$ can be moved to the output in
sorted order, using the operations corresponding to the one-way sections
of track. If such a sequence of operations exists, we say that the
permutation $\pi$ can be sorted by the given network.

Some networks are very restrictive in the set of permutations which
they can sort. For example, the only permutation which can be sorted
on a single queue is the identity permutation itself. Alternatively,
it is also possible to construct switchyard networks which are capable
of sorting arbitrary input permutations. Between these two extremes,
however, is a rich variety of sorting capabilities.

In considering the question of which permutations could be computed
on certain switchyard networks (where a permutation is computable
on a network $\mathcal{N}$ if and only if its inverse is sortable
on $\mathcal{N}$), Vaughan Pratt showed that any switchyard network
$\mathcal{N}$ capable of sorting a permutation $\pi$ must also be
capable of sorting every permutation contained in $\pi$, where containment
is defined in the following way. \cite{pratt}

\begin{quotation}
$\overset{def}{=}$ A permutation \underbar{$\pi\in S_{n}$ contains
the permutation $\sigma\in S_{k}$} if and only if $\sigma$ can be
recovered from $\pi$ by removing a (possibly empty) subset of its
elements, and then reducing the values of the remaining elements as
necessary to remove any gaps (so that they consist of exactly the
set $\left\{ 1,2,\ldots,k\right\} $). If $\pi$ does not contain
$\sigma$, we say that \underbar{$\pi$ avoids $\sigma$}.
\end{quotation}

This permutation containment relation is sometimes denoted $\sigma\preceq\pi$,
and it is easy to see that it creates a poset on the set of all permutations.
Pratt showed that the set of permutations which are sortable on a
given switchyard network, viewed as a subset of the poset of all permutations,
is closed under downward containment. Sets of permutations having
this property have since become a major research area, and have been
given their own title.

\begin{quotation}
$\overset{def}{=}$ A \underbar{permutation class}, $\mathcal{C}$,
is a set of permutations such that $\sigma\in\mathcal{C}$ whenever
$\sigma\preceq\pi$ and $\pi\in\mathcal{C}$.
\end{quotation}

Another way of defining a permutation class (also dating back to Pratt),
is to consider the set of minimal permutations not contained in that
class. Such a set, called the basis of $\mathcal{C}$ and denoted
$\text{Bas}\left(\mathcal{C}\right)$, can be used to determine whether
a given permutation is contained in the class $\mathcal{C}$. If $\pi$
contains some element of $\text{Bas}\left(\mathcal{C}\right)$, then
it cannot be in $\mathcal{C}$, since that would imply that every
permutation contained in $\pi$ must also be in $\mathcal{C}$ by
the definition of a permutation class. Conversely, if $\pi$ doesn't
contain any element of the basis then it must be in $\mathcal{C}$,
since otherwise either it or some permutation contained in it must
be a minimal permutation not in $\mathcal{C}$.

The basis of a permutation class is clearly an antichain in the poset
of all permutations. Furthermore, by considering basis with infinite
size, any permutation class can be described by the set of basis permutations
which it avoids. This description of permutation classes as sets of
permutations avoiding certain sets of basis permutations is now the
standard representation. We notate such a class as $\mathcal{C}=\text{Av}\left(\text{Bas}\left(\mathcal{C}\right)\right)$,
and the permutations of length $n$ in $\mathcal{C}$ by $\mathcal{C}_{n}=\text{Av}_{n}\left(\text{Bas}\left(\mathcal{C}\right)\right)$.

\begin{example}
Consider the permutation class which consists of permutations sortable
on a single stack, $\mathcal{C}$. Knuth showed that a permutation
is sortable on a single stack if and only if it avoids the pattern
$231$. Therefore $\mathcal{C}=\text{Av}\left(231\right)$.
\end{example}

In this work, we are interested in two specific permutation classes.
The class of permutations which are sortable on two stacks in parallel,
$\mathcal{C}$, and the class of permutations which are sortable on
a double-ended queue (also called a deque), $\mathcal{D}$. These
are two of the classes which Pratt investigated in his 1973 paper,
and he was able to find the basis of both of these classes.
In each case, the basis is an infinite set which can be described
by the pattern used to construct basis elements of each length.

The basis for the class of parallel stack sortable permutations, $\mathcal{C}$,
consists of permutations having length greater than $3$ and equivalent
to $0$ or $3$ modulo $4$, which fall in the following pattern\begin{align*}
 & 2\ 3\ 4\ 1\\
 & 5\ 2\ 7\ 4\ 1\ 6\ 3\\
 & 2\ 7\ 4\ 1\ 6\ 3\ 8\ 5\\
 & 9\ 2\ 11\ 4\ 1\ 6\ 3\ 8\ 5\ 10\ 7\\
 & 2\ 11\ 4\ 1\ 6\ 3\ 8\ 5\ 10\ 7\ 12\ 9\\
 & \vdots\end{align*}

Similarly, the basis of the class of deque sortable permutations,
$\mathcal{D}$, consist of four permutations of each odd length greater
than $4$. One representative of each set of four falls in the following
pattern\begin{align*}
 & 5\ 2\ 3\ 4\ 1\\
 & 5\ 2\ 7\ 4\ 1\ 6\ 3\\
 & 9\ 2\ 7\ 4\ 1\ 6\ 3\ 8\ 5\\
 & 9\ 2\ 11\ 4\ 1\ 6\ 3\ 8\ 5\ 10\ 7\\
 & 13\ 2\ 11\ 4\ 1\ 6\ 3\ 8\ 5\ 10\ 7\ 12\ 9\\
 & 13\ 2\ 15\ 4\ 1\ 6\ 3\ 8\ 5\ 10\ 7\ 12\ 9\ 14\ 11\\
 & \vdots\end{align*}
The other three basis patterns of each length can be recovered by
some combination of interchanging the first two elements of the permutation,
and interchanging the largest two elements of the permutation.

Notice that every odd length pattern from $\text{Bas}\left(\mathcal{C}\right)$
is represented in $\text{Bas}\left(\mathcal{D}\right)$. This should
not surprise us, since sortability on parallel stacks and on a deque
are closely linked concepts. In fact, we can view a deque switchyard
network as being just a parallel stack switchyard network in which
the bottoms of the two stacks have been joined together to form a
single linear storage element. Clearly, the permutations which are
sortable on a deque are a superset of of those sortable on parallel
stacks.

When we set about the investigation leading to this work, our interest
was primarily in the permutation class $\mathcal{D}$, the permutations
which are sortable on a deque. However, we address $\mathcal{C}$
as well, since most of our results for $\mathcal{D}$ contain simplifications
which pertain to $\mathcal{C}$.

\section{The Enumeration Problem}
One question can be asked about a given permutation class is,
``how many permutations of length $n$ are in the class''? Even
though Pratt provided a full desription of the permutation classes
$\mathcal{C}$ and $\mathcal{D}$ by giving their basis patterns,
such a description says almost nothing about the number of permutations
in the classes of various sizes.

When presented with the task of enumerating a sequence, such as the
number of permutations in a given permutation class having length
$n$ for $n=1,2,3,\ldots$, there are several different forms that
an answer can take.

The most satisfying answer would be an explicit closed form formula
as a function of $n$. For example, it has been shown that the number
of permutations of length $n$ which can be sorted on a single stack
is the $n$th Catalan number. (That is, $\left|\text{Av}_{n}\left(231\right)\right|=C_{n}$.)

Another desirable answer is a generating function whose coefficients
count the desired sequence. Generating functions have been found for
several permutation classes for which closed form formulas are not
known.

Without a closed formula or a generating function, one is left with
asymptotic analysis for an inexact view of the long term behavior
of the sequence, and with algorithms for calculating the $n$th term
for an exact view of a limited number of terms at the beginning of
the sequence.

The problem of enumerating the sequences $\left|\mathcal{C}_{1}\right|,\left|\mathcal{C}_{2}\right|,\ldots$
and $\left|\mathcal{D}_{1}\right|,\left|\mathcal{D}_{2}\right|,\ldots$
with a closed form solution or a generating function has gone unsolved
for 40 years. \cite{bona} Much of the work that has
been done in the enumeration of these two classes has been devoted
to studying their asymptotic behavior.

We know that every permutation class having a nonempty basis has a
growth rate that is at most exponential. Furthermore, every permutation
class, $\mathcal{B}$, describing permutations which are sortable on
some switchyard network is supermultiplicative (the number of sortable
permutations of length $m+n$ is greater than or equal to the product of the number
of sortable permutations of length $m$ and the number of sortable
permutations of length $n$), which implies that the limit
\begin{align*}
\lim_{n\rightarrow\infty}\left|\mathcal{B}_{n}\right|^{\frac{1}{n}}
\end{align*}
is well defined. This limit is know as the growth rate of the permutation
class and is denoted $\text{gr}\left(\mathcal{B}\right)$. The sequence
enumerating the number of in-class permutations of each length then
grows like $\left(\text{gr}\left(\mathcal{B}\right)\right)^{n}$.

Neither $\text{gr}\left(\mathcal{C}\right)$ nor $\text{gr}\left(\mathcal{D}\right)$
is known exactly, but the best known bounds, found by Albert, Atkinson, and Linton
in 2009, give a very good estimate of what these growth rates must
be:

\begin{center}
\begin{tabular}{|c|c|c|}
\hline 
 & lower bound & upper bound\tabularnewline
\hline
\hline 
$\text{gr}\left(\mathcal{C}\right)$ & 7.535  & 8.3461\tabularnewline
\hline 
$\text{gr}\left(\mathcal{D}\right)$ & 7.890 & 8.352\tabularnewline
\hline
\end{tabular}\cite{albert}
\par\end{center}

Notice that, asymptotically, the number of permutations which are
sortable on parallel stacks must be very close to the number of permutations
which are sortable on a deque. In fact, Albert et al. have conjectured
that the growth rates of these permutations may be equal.

In contrast to the investigation of the growth rates of these permutation
classes, it seems that comparatively less work has been done on the
problem of developing algorithms to calculate the terms of the sequence
explicitly. The first twelve terms of $\left|\mathcal{D}_{1}\right|,\left|\mathcal{D}_{2}\right|,\ldots$
were known to Flajolet, Salvy, and Zimmermann in 1989 \cite{flajolet}.
In April of 2012, Zimmermann posted the first fourteen terms of this
sequence on the online encyclopedia of integer sequences (http://oeis.org/A182216),
along with a C program designed to compute these terms. Zimmermann's
program works by constructing words out of the alphabet of operations
available to the deque switchyard (the alphabet $\left\{ a,b,y,z\right\} $)
and determining which permutations are sorted by these words. Zimmermann
uses some relations in order to avoid enumerating all $16^{n}$ possible
words of length $2n$, but even so, this approach has an exponential
runtime whose base is strictly greater than $\text{gr}\left(\mathcal{D}\right)$. 

The problem of enumerating the sequence $\left|\mathcal{C}_{1}\right|,\left|\mathcal{C}_{2}\right|,\ldots$
is even less well known, and to the best of our knowledge, there are
no known algorithms for computing it in less than $\omega\left(\text{gr}\left(\mathcal{C}\right)\right)$.
We do not know how many terms of this sequence are currently known.

In this work, we will provide two new algorithms for computing the
leading terms of these sequences. The first, which employs a parallel 
stack/deque sortability testing algorithm by Rosenstiehl and Tarjan has a
runtime of $\Theta\left(n^{2}X^{n}\right)$ where $X$ is equal to
the growth rate of the relevant class. Modulo the sub-exponential factor
$n^{2}$, this is optimal among algorithms which must consider each
sortable permutation. By harnessing symmetries inherent in the execution
of the Rosenstiehl-Tarjan algorithm, however, we have developed a
second algorithm with a runtime of $O\left(n^{5}2^{n}\right)$. The
next several sections of this work are devoted to discussion of these
algorithms.

\section{The Rosenstiehl-Tarjan Algorithm for Parallel Stacks}

Suppose we are given some permutation $\pi$, and we wish to determine
whether $\pi$ belongs to the permutation class $\mathcal{C}$. We
call this the membership testing problem. In 1982, Rosenstiehl and
Tarjan presented an algorithm which can answer this question in linear time ($O\left(n\right)$ where $n$ is the length of $\pi$). \cite{tarjan} (This runtime
is optimal among approaches which must read a constant fraction of
the permutation $\pi$ to determine its membership.) Rosenstiehl and
Tarjan's algorithm works by using a data structure, which they call
a pile of twinstacks, which simultaneously records all possible configurations
of the parallel stack switchyard network throughout the process of trying
to sort $\pi$. Since we make extensive use of Rosenstiehl and Tarjan's
algorithm, we present it here in its entirety. We begin with a definition
of the fundamental data-structure unit used by the algorithm.

\begin{quotation}
$\overset{def}{=}$ Let a \underbar{twinstack}, $\left[L,R\right]$,
be a pair of stacks, called the left stack and the right stack, each
of which contains permutation elements in strictly increasing order
from top to bottom. A proper twinstack must always have at least one
of its stack nonempty.
\end{quotation}

The Rosenstiehl and Tarjan algorithm represents the current state
of two parallel stacks by a stack of twinstack, which Rosenstiehl
and Tarjan call a pile of twinstacks. Each twinstack in the pile can
be subject to several operations. A \underbar{reversal} swaps the
left and right stacks. A \underbar{weld} combines the top two twinstacks
into a single twinstack by concatenating their left stacks to form
the new left stack, and by concatenating their right stacks to form
a new right stack.

Clearly, we cannot allow welding in cases where a larger element would
be concatenated on top of a stack containing a smaller element (since
this violates our definition of a twinstack). In fact, in general
we would like to maintain the even stronger condition that each element
contained in a given twinstack is smaller than every element
in every twinstack below that twinstack. A pile of twinstacks for
which this property holds is called \underbar{normal}.

The intuition behind the Rosenstiehl-Tarjan algorithm is that each
twinstack represents a degree of freedom in the positioning of the
elements among the two parallel stacks. A normal pile of $k$ twinstacks
represents $2^{k}$ different configurations of parallel stacks.  These
can be recovered by choosing one of $2^{k}$ different subsets of
the twinstacks in the pile and reversing them, and then welding down
the entire pile. (By welding down the pile, we mean applying successive
weld operations to the top pair of twinstacks on the pile until only
a single twinstack remains.)

The Rosenstiehl-Tarjan algorithm processes the elements of the permutation
in order, maintaining a normal pile of twinstacks which represents
all of the elements currently received from the input but not yet
sent to the output. In each iteration of the algorithm, we first receive
the next element, $i$, from the input and place it in its own twinstack
on top of the pile (the twinstack $\left[i,-\right]$). We then attempt
to normalize the pile, in case the addition of this new twinstack
caused the pile to no longer be normal.

Normalization Step: Notice that the only element which could possibly
be larger than any element in a lower twinstack is the new element
$i$ (since the pile would have been normalized during the previous
iteration).  Thus, we compare the element $i$ to the top element(s) of the
stacks of the second twinstack, resulting in one of the following
cases:
\begin{itemize}
\item If $i$ is smaller than any top elements in the second twinstack,
then the pile is already normalized, so return from the normalization
step.
\item If $i$ is smaller than one top element of the second twinstack, but
larger than another top element, reverse the twinstack containing
$i$ in order to position it over the side of the second twinstack
whose top element is larger than $i$ and then weld.
After the weld, the new top twinstack contains an element $j$ which
is larger than $i$. Since we know that $j$ is not larger than any
elements in lower twinstacks, the whole pile must be normalized. Return
from the normalization step.
\item If $i$ is larger than one top element of the second twinstack and
the other side of the second twinstack is empty, reverse the twinstack
containing $i$ in order to position it over the empty side of the
second twinstack and then weld it. At this point, $i$ may or may
not be larger than some element in the new second twinstack (previously
the third twinstack), so repeat the normalization step.
\item If $i$ is larger than both top elements of the second twinstack,
abort the algorithm (the given permutation is not sortable).
\end{itemize}
Finally, after successfully normalizing the pile, we examine the top elements of
the top twinstack and move one of them to the output if it is the next
element belonging there. If this causes the only nonempty stack of
the top twinstack to become empty, remove it. Repeat this process
until no more elements can be moved to the output.

An execution of the Rosenstiehl-Tarjan algorithm has two possible
results. One possibility is that the the algorithm returns false because
it was at some point unable to normalize the pile of twinstacks. (Intuitively,
this corresponds to determining that two elements $j$ and $k$ must
be on opposite stacks (as represented by having them on opposite sides
of the same twinstack), both of which are still in the stacks when
a new larger element $i$ arrives from the input. Whichever stack
$i$ is placed on, it must necessarily pin $j$ or $k$ underneath
it, preventing that element from ever making it to the output.) The
second possibility is that, after $n$ iterations of the algorithm,
every element has been moved from the input and into the output in
sorted order. In this case the algorithm returns true.

It is worth noting that we can also recover the sequence of operations
for sorting a permutation which the Rosenstiehl-Tarjan algorithm deems
sortable. However, for our purposes we are only interested
in the boolean result telling whether or not the given permutation
is sortable.

\section{The Rosenstiehl-Tarjan Algorithm for Deques}

After the main results of their paper, Rosenstiehl and Tarjan also provided, as an aside,
 a modification of their algorithm to allow testing sortability on deques.
They write:

\begin{quotation}
We use the same algorithm as in the case of twin stacks, except
that we process an element $i$ larger than anything on the
{[}pile of twinstacks{]} as follows. Add a new twinstack $\left[i,-\right]$
to the \emph{bottom} of the {[}pile of twinstacks{]}. If any twinstack
$\left[L_{i},R_{i}\right]$ has both $L_{i}$ and $R_{i}$ nonempty,
abort. Otherwise, reverse as necessary to make all the $R_{i}$s empty,
and weld all the twinstacks in the {[}pile of twinstacks{]}. \cite{tarjan}
\end{quotation}

The intuition behind this modification is simple. Whenever we add
a new element from the input to a deque, it is safe to add that element
as long as the deque can be arranged in monotonic order. Clearly,
if all of the twinstacks on the pile have one empty side, then there
is an arrangement of elements which is monotonic. If the standard
parallel stack sortability algorithm was used to add a new maximal
element $i$ at this point, this would result in the entire pile of
twinstacks being welded together with $i$ on one side, and all of
the previous contents of the twinstack on the other. Clearly the deque
is still monotonic, but this is no longer evinced by the absence
of double-sided twinstacks. Thus if we were to subsequently
add another larger element, we would fail during the normalization
step. 

Rosenstiehl and Tarjan's approach, therefore, is to essentially tuck
the element $i$ underneath the side stack containing all of the other
elements after welding. Thus the pile continues to contain only one-sided
twinstacks. 

We have found, however, that there is a small error in the modification
that Rosenstiehl and Tarjan give in their paper. Notice that their
algorithm will return false whenever a new maximal element is
received from the input at a time when the pile contains a double-sided
twinstack. Thus their algorithm depends on the invariant that the
pile never contains a double-sided twinstack as long as it can represent
a monotonic deque state.

Whenever a normalized pile contains a double-sided twinstack apart from
the bottom twinstack, then every possible state of the deque
must necessarily be non-monotonic. Similarly, if the bottom twinstack
is double-sided, and both sides contain more than one element, or
there is an element which is larger than both top elements, then the
deque must necessarily be non-monotonic. Therefore, the problem case
we must watch to avoid is where only the bottom twinstack is double-sided,
and one side of this twinstack contains a single element larger than
every element on the other side. \emph{This is the only case where the pile
can contain a double-sided twinstack while simultaneously representing
a monotonic deque state.}

The special treatment that the modification gives to the introduction
of a new maximal element ensures that the algorithm never creates
a double-sided twinstack so long as the deque remains monotonic. However,
this does not protect against the case where the popping of an element
to the output causes a deque to become monotonic. It is possible that,
in a non-monotonic state, the pile can contain a double-sided twinstack.
Then, by popping an element to the output, the state can become monotonic
without removing the double-sided twinstack.

To give a concrete example, consider the permutation $254163$. This
is clearly deque-sortable:

\begin{figure}[ht]
\center{\includegraphics[scale=0.33]{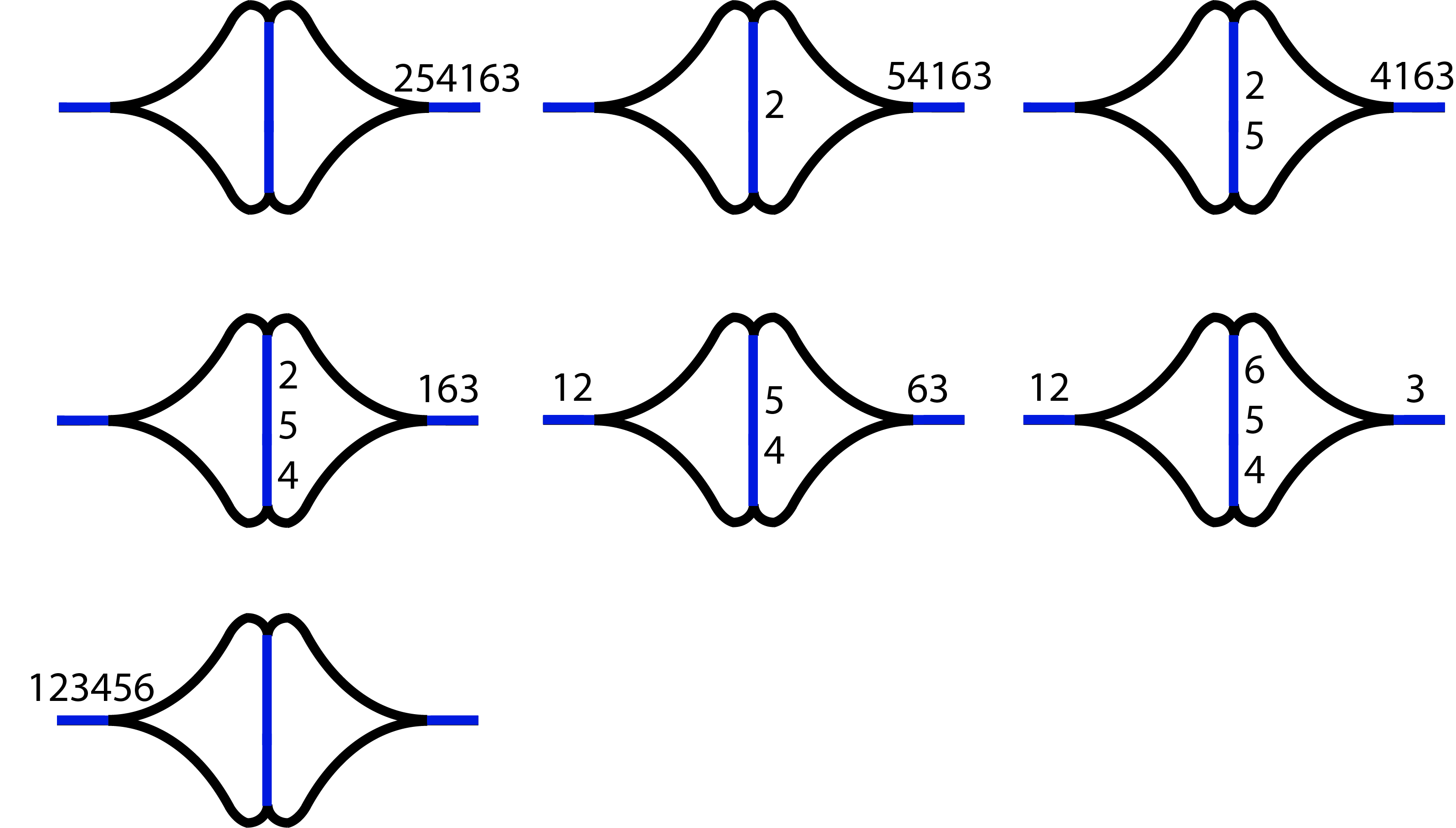}}
\caption{Sorting 254163}
\end{figure}

Now consider running this through the stated version of the algorithm:

\begin{center}
\begin{tabular}{|c|c|c|}
\hline 
Output: & Pile of twinstacks: & Input:\tabularnewline
\hline
\hline 
 &  & $254163$\tabularnewline
\hline 
 & $\begin{array}{c}
\\\left(\begin{array}{cc}
2 & \end{array}\right)\\
\\\end{array}$ & $54163$\tabularnewline
\hline 
 & $\begin{array}{c}
\\\left(\begin{array}{cc}
2 &\\
5 &\end{array}\right)\\
\\\end{array}$ & $4163$\tabularnewline
\hline 
 & $\begin{array}{c}
\\\left(\begin{array}{cc}
2\\
5 & 4\end{array}\right)\\
\\\end{array}$ & $163$\tabularnewline
\hline 
$12$ & $\begin{array}{c}
\\\left(\begin{array}{cc}
4 & 5\end{array}\right)\\
\\\end{array}$ & $63$\tabularnewline
\hline 
 & Abort! & \tabularnewline
\hline
\end{tabular}
\par\end{center}

The problem here is clearly that popping the element $2$ to the output
results in the state becoming monotonic, even though the bottom (and
only) twinstack remains double-sided. The fix for this problem is
very simple. Whenever we pop an element from the bottom twinstack,
we need to check if the resulting state is monotonic. If it is, we
rearrange the bottom twinstack as necessary (by tucking the largest
element at the bottom of the stack containing the other element) to
make it one-sided.

\section{Correctness of Rosenstiehl-Tarjan-Modified}
Since Rosenstiel and Tarjan did not give a full proof of correctness
of their algorithm for testing sortability on a deque, and since we
have shown that some modifications need to be made to to fix this
algorithm, it should be worthwhile to take the time to fully prove
the correctness of the new version of the algorithm which we will call
Rosenstiehl-Tarjan-Modified. We intend to prove this by considering
a mapping relating the states of a run of Rosenstiehl-Tarjan-Modified
to the states of a sorting run on an actual deque. Therefore, we will
begin by examining what these states are.

The \underbar{state} of a deque switchyard can be thought of as consisting of
three lists. The first list is the output, which contains all the elements
which have already been popped from the deque in the order in which
they were popped. The second list is the deque itself. This list contains
the elements which, at the current point in the run, have already been
pushed from the input onto the deque, but have not yet been popped.
The final list comprising the deque switchyard state is the input
list. This list contains a suffix of the input permutation consisting
of all those elements which have not yet been pushed onto the deque.

At all times, the combined three lists of the deque state contain all $n$
elements of the input permutation $\pi$. The transition
rules are governed by the four operations allowed in sorting on a
deque. Whenever the input list is nonempty, we are allowed to take
operation $a$, by removing the first element of the input list and
adding it to the left end of the deque list. Alternatively, we can
make a state transition by taking operation $b$: removing the first
element of the input list and adding it to the right end of the deque
list. The other two operations, $y$ and $z$, involve removing the
left or right end element of a nonempty deque list, and placing the
removed element at the end of the output list.

Let $\mathfrak{D}$ be the set of all states of a deque switchyard
containing $n$ elements. Then each of the operations $a,b,y,$ and
$z$ defines a map on a subset of $\mathfrak{D}$ into $\mathfrak{D}$.
Alternatively, we can view $\mathfrak{D}$ as the vertex set of a simple acyclic directed
graph, where each vertex has between zero and four out-edges, labeled with 
the operations from $\left\{a,b,y,z\right\}$ corresponding to the represented transitions.  
Note that some edges may have multiple labels since, for example, the 
operations $a$ and $b$ correspond to the same state transition
whenever the deque list is empty.  Alternatively, some operations may not be
represented among the labels on the out-edges from some nodes.
This is the case for the operations $y$ and $z$ for any state whose
deque list is empty.  Also note that this graph represents all possible states
for a deque switchyard containing $n$ elements.  This includes many
states in which the output contains elements which are out of order.

We speak of a \underbar{run
of a permutation $\pi$} on a deque switchyard to refer to a walk on
this directed graph, starting at the the state where the output and
deque are empty and the input list contains the full permutation $\pi$.
The run then consists of a series of states connected with edges
each of which is labeled with at least one of the operations $a,b,y,$ and
$z$. A run is successful if
it takes $2n$ steps and then ends at the unique state containing
the identity permutation $1\ldots n$ in the output list. A permutation $\pi$
is sortable on a deque if and only if there is a successful run of
that permutation on a deque.

Since the out-edges from each vertex of $\mathfrak{D}$ are each labeled
with a nonempty subset of $\lbrace a,b,y,z \rbrace$, for each run of a permutation $\pi$ on a deque switchyard we can
can construct a corresponding word in the alphabet $\lbrace a,b,y,z\rbrace $ by choosing 
one letter from the label set for each edge along the run.
Furthermore, given such a word, we can easily determine whether
it corresponds to a valid run (in terms of only
selecting operations which are available at a given state). A word
in the alphabet $\left\{ a,b,y,z\right\} $ represents a valid run
of a permutation $\pi$ or length $n$ if and only if it contains
at most $n$ combined occurrence of the letters $a$ and $b$, and
every prefix of the word contains at least as many combined occurrence
of $a$ and $b$ as of $y$ and $z$. 

\begin{quotation}
$\overset{def}{=}$ We call a run of the permutation $\pi$ on
a deque switchyard \underbar{reduced} if every state in which an element
$i$ at one of the ends of the deque is the next element required by
the output is followed by a state in which that element $i$ has been
popped from the deque and moved to the output.
\end{quotation}

We are interested in reduced runs because it is convenient to design
our algorithms to only explore reduced runs, by sequentially moving
one element from the input to the deque, and then moving elements
from the deque to the output until we are unable to continue to do
so. The following lemma shows that this choice is not restrictive.

\begin{lem}
Every permutation $\pi$ which can be sorted on a deque can be sorted
using a reduced run on a deque.
\end{lem}

\begin{proof}
Suppose that we are given some permutation, $\pi$, that is sortable
on a deque. Then there exists some successful run sorting that permutation.
Call this run $r$.

Let $\omega$ be a word in the alphabet $\left\{ a,b,y,z\right\} $
which corresponds to $r$. (Recall that there can be multiple such
words, but there must always be at least one such word.) Suppose that
$r$ is not a reduced run. Then there is some element, $i$, which
is not moved to the output as soon as possible. However, since $r$
is a successful run, $i$ must be moved to the output eventually. 

Let $\omega_{j}$ be the letter in the word $\omega$ which corresponds
to the state transition wherein $i$ is moved to the output. $\omega_{j}$
must be either $y$ or $z$. Suppose that $\omega_{j}=y$. Consider
the first opportunity to move $i$ to the output. At no point in-between
then and the step corresponding to $\omega_{j}$ can there be any
element to the left of $i$. This is because, as soon as we are ready to
move $i$ to the output, all elements from $1$ through $\left(i-1\right)$
have already been moved to the output, so there will not be any $y$
or $z$ operations preceding the one which outputs $i$.  Since the left 
side of $i$ is free at the time corresponding to $\omega_{j}$, it must have been
free for the entire intervening period. Thus we can construct a new run, $r'$, 
by moving the $y$
operation which takes $i$ to the output to the earliest possible
opportunity.  (The same result holds for $\omega_{j}=z$ by symmetry.)

The new run, $r'$, has one fewer elements which is not popped at
the first opportunity than $r$ did. We can repeat this process until
we arrive at a run which has no elements which are not popped at the
first opportunity. This resulting run is reduced. Therefore, the arbitrary
sortable permutation $\pi$ can be sorted using a reduced run. 
\end{proof}

We would also like to consider what can cause a run to not be successful.
Notice that, whenever the state of the deque switchyard is such that
there is an element $i$ on the deque which is sandwiched between
two larger element, then there is no successful run including this
state. This is because one of the two sandwiching elements must be
moved to the output before $i$ can be moved, but this will result
in the output being unsorted. We call such states sandwich states,
and we seek successful runs among those runs which avoid these sandwich
states.

We are now ready to consider the states of the Rosenstiehl-Tarjan-Modified
algorithm, and the mapping which takes them to reduced runs on the
deque. The state of Rosenstiehl-Tarjan-Modified also consists of three
parts. The output is again a list of elements which have been popped
in the order in which they were popped. The input is again a list
containing a suffix of $\pi$ with all those elements which have not
yet been pushed. Instead of a deque list, however, the third element
of the Rosenstiehl-Tarjan-Modified state is the stack of twinstacks called the pile.
Like the states of
the deque switchyard, the stacks and lists of the Rosenstiehl-Tarjan-Modified
algorithm always contain all $n$ permutation elements. Let $\mathfrak{R}$
denote the set of all possible states of the algorithm.

Here is the psuedocode of the Rosenstiehl-Tarjan-Modified algorithm:

\begin{minipage}{\linewidth} 
\begin{algorithm}[H]
\SetAlgoRefName{FromInput$(O,P,I)$}
\DontPrintSemicolon   
\ProcSty{Procedure} \FromInput{$O,P,I$}\;

$x=I.dequeue()$\;
add the new twinstack $(x,-)$ to the top of $P$\;
\Return{}
\end{algorithm} 
\end{minipage}

\begin{minipage}{\linewidth} 
\begin{algorithm}[H]
\SetAlgoRefName{Normalize$(O,P,I)$}
\DontPrintSemicolon   
\ProcSty{Procedure} \Normalize{$O,P,I$}\;
$topTStack=P.top()$\;
$secondTStack=P.second()$\;
\If{secondTStack==NIL}{
  \Return{TRUE}
}
$x=topTStack.left().top()$\;
\tcc{note that $x$ is the only element in $topTStack$ which can possibly be greater than some element in $secondTStack$}
\Switch{}{
  \uCase{secondTStack.right() nonempty and x less than both top elements of secondTStack}{
    \Return{TRUE}
  }
  \uCase{secondTStack.right() nonempty and x in-between the two top element of secondTStack}{
    weld down\;
    \Return{TRUE}
  }
  \uCase{secondTStack.right() nonempty and x greater than both top elements of secondTStack}{
    \Return{FALSE}
  }
  \uCase{secondTStack.right() empty and x less than secondTStack.left().top()}{
    \Return{TRUE}
  }
  \Case{secondTStack.right() empty and x greater than secondTStack.left().top()}{
    \If{secondTStack is the bottom twinstack and x is larger than every element in secondTStack}{
      place $x$ at the bottom of the nonempty side of $secondTStack$\;
      reverse $topTStack$ and weld down\;
    }
    \Else{
      reverse $secondTStack$ and weld down\;
      \Return{\Normalize{$O,P,I$}}
    }
  }
}
\end{algorithm} 
\end{minipage}

\begin{minipage}{\linewidth} 
\begin{algorithm}[H]
\SetAlgoRefName{ToOutput$(O,P,I)$}
\DontPrintSemicolon   
\ProcSty{Procedure} \ToOutput{$O,P,I$}\;

$topTStack=P.top()$\;
\If{topTStack.right().top()==O.last()+1}{
  $O.enqueue(topTStack.right().pop())$\;
  reverse $topTStack$ if necessary to put the largest top element on the left stack\;
  \If{the pop was from the bottom twinstack and it caused that twinstack to become monotonic}{
    reorganize $topTStack$ to be one-sided, reflecting its monotonicity\;
  }
  \ToOutput{$O,P,I$}\;
  \Return{}
}
\ElseIf{topTStack.left().top()==O.last()+1}{
  $O.enqueue(topTStack.left().pop())$\;
  pop $topTStack$ if it is now empty\;
  otherwise reverse $topTStack$ if necessary to put the largest top element on the left stack\;
  \ToOutput{$O,P,I$}\;
  \Return{}
}
\Else{
  \Return{}
}
\end{algorithm} 
\end{minipage}

\begin{minipage}{\linewidth} 
\begin{algorithm}[H]
\SetAlgoRefName{RosenstiehlTarjanModified$(O,P,I)$}
\DontPrintSemicolon   
\ProcSty{Procedure} \RosenstiehlTarjanModified{$O,P,I$}\;
\While{{\bf not} I.isEmpty()}{
  \FromInput{$O,P,I$}\;
  \If{{\bf not} \Normalize{$O,P,I$}}{
    \Return{FALSE}
  }
  \ToOutput{$O,P,I$}\;
}
\Return{TRUE}
\end{algorithm} 
\end{minipage}

As discussed previously, the idea behind the Rosenstiehl-Tarjan-Modified
algorithm is that we simultaneously represent several possible states
of a deque sorting attempt by representing degrees of freedom with
the ability to reverse each twinstack on the pile independently. Each
twinstack is meant to hold elements which belong ``outside of''
the elements in the twinstacks below it. If we were to construct
an actual deque list from the pile of twinstacks, we would choose
an orientation for each twinstack. Starting with the empty deque
and the bottom twinstack, we add the elements from the left side of
the twinstack to the left side of the deque and the elements from
the right side of the twinstack to the right side of the deque if
we selected the default orientation. Alternatively, if we select
the reversed orientation, then we put the left stack elements on
the right side of the deque and the right stack element on the left
side of the deque. Clearly, the orientation that is chosen for the
bottom twinstack may or may not matter (depending on whether
the bottom twinstack contains one or several elements), but each change of orientation
for a non-bottom twinstack results in a different deque state.

We use the term \underbar{realization} to refer to this process of choosing orientations
for the twinstacks in the pile and turning them into a deque. Every
pile containing $2^{k}$ twinstacks can be realized as either $2^{k}$
or $2^{k-1}$ different deques, depending on whether the bottom twinstack
contains exactly one permutation element. Since we require every twinstack
to be nonempty, there can only be a maximum of $n$ twinstacks (and
if there are $n$ twinstacks, then the bottom twinstack contains exactly
one element). Thus every pile can be realized by at most $2^{n-1}$
different deques. This process of realization is the mapping which
we will use to prove the correctness of the Rosenstiehl-Tarjan-Modified
algorithm. Let $\phi$ be a mapping
\begin{align*}
\phi:\left(\mathbb{Z}/2\mathbb{Z}\right)^{n-1}\times\mathfrak{R}\rightarrow\mathfrak{D}\end{align*}
defined by choosing an orientation for every stack starting with the
bottommost stack whose orientation matters (the bottommost stack if it has two elements and the second from the bottom otherwise), and then combining
their elements as described into a single deque list.

\begin{quotation}
$\overset{def}{=}$ We call $\phi$ the \underbar{realization mapping}.  We
say that a state $d\in\mathfrak{D}$ \underbar{is a realization of} $r\in\mathfrak{R}$
if there exists some $\alpha\in\left(\mathbb{Z}/2\mathbb{Z}\right)^{n-1}$
such that $d=\phi\left(\alpha,r\right)$.
\end{quotation}

Just like we define runs of a deque switchyard as walks on a graph
with vertices in $\mathfrak{D}$, we would like to define runs of the
Rosenstiehl-Tarjan-Modified algorithm to be walks on a simple acyclic directed graph with
vertices in $\mathfrak{R}$. The edges, in this case, represent the states 
transitions that can be accomplished during the course of execution of the
Rosenstiehl-Tarjan-Modified algorithm.  (Since the algorithm is deterministic,
every vertex has at most one out-edge, and the entire run is determined
by the choice of starting vertex.)  Once again, a successful run will be a
walk starting with the permutation $\pi$ in the input list, and ending
after $2n$ steps with the identity permutation in the output list.

We now generalize the notion of a realization from single states to entire runs.

\begin{quotation}
$\overset{def}{=}$ We say that a deque run $d_{0},\ldots,d_{k'}$ is
\underbar{a realization of} a Rosenstiehl-Tarjan-Modified run $r_{0},\ldots,r_{k}$
if it is a valid deque run, and there exists a sequence
of binary numbers $\alpha_{0},\ldots,\alpha_{k}\in\left(\mathbb{Z}/2\mathbb{Z}\right)^{n-1}$ such that the sequence $\phi\left(\alpha_{0},r_{0}\right),\ldots,\phi\left(\alpha_{k},r_{k}\right)$ is equal to $d_{0},\ldots,d_{k'}$ except possibly with repetitions of states.
\end{quotation}

\begin{lem}
If $r=r_{0},r_{2},\ldots,r_{k}$ is any run of the Rosenstiehl-Tarjan-Modified
algorithm starting with $\pi$ in the input and ending after $k$
steps at $r_{k}\in\mathfrak{R}$, then for any $\alpha\in\left(\mathbb{Z}/2\mathbb{Z}\right)^{n-1}$
there exists a realization of $r$ which ends at the state $\phi\left(\alpha,r_{k}\right)\in\mathfrak{D}$.
\end{lem}

\begin{proof}
We give a proof by induction on $k$.

base case ($k=0$): Clearly the one and only realization of the Rosenstiehl-Tarjan-Modified run
of trivial length starting with input $\pi$ is mapped
to the one and only valid deque run of trivial length starting
with that input $\pi$.

inductive case: Assume that the statement holds for $k-1$ steps.
Consider the $k$th step of run $r$. The Rosenstiehl-Tarjan-Modified
algorithm can bring about change in its state in a couple of ways.
\begin{enumerate}
\item A new twinstack could be added to the top of the pile by the FromInput
subroutine.
\item An element can be moved from the pile to the ouput (possibly with
the removal of the containing twinstack).
\item The top twinstack could be welded down.
\item A twinstack could be reversed.
\item The single largest element of the bottom twinstack, which is currently
residing as the only element in the stack in one side of that twinstack,
can be moved to the bottom of the other stack of that twinstack. (This
can occur in the ToOuput and Normalize subroutines.)
\end{enumerate}
The last two possibilities, reversing a twinstack and rearranging the
bottom twinstack, do not affect which states of the deque switchyard
can be realized. Therefore, the exact same run of the deque switchyard
which realizes $r_{0},\dots,r_{k-1}$ can realize $r_{0},\ldots,r_{k}$
by repeating the last state of the realization. 

In the third possibility, where the top twinstack is welded, the states
of the deque switchyard which are realizable from $r_{k}$ are a subset
of those realizable from $r_{k-1}$. Thus, for every realization of
$r_{k}$, there is already a run of the deque realizing $r_{0},\ldots,r_{k-1}$
which ends at that state. This can be extended to realize $r_{0},\ldots,r_{k}$
by repeating the last state of the realization.

In case 2, we are moving one element, $x$, from the pile to the output.
By the definition of realization, it is clear that every realization
of $r_{0},\dots,r_{k-1}$ ends with a deque switchyard state having
the element $x$ as one of the ends of the deque. 
Let $S_{\text{pop}}\subseteq\mathfrak{D}$ denote the set of states which
can be transitioned to from realizations of $r_{k-1}$ via a pop operation 
sending $x$ to the output.  If any of these is not a realization of $r_{k}$, 
then the state it is a transition from must not be a realization of $r_{k-1}$, a contradiction.
Thus $S_{\text{pop}}$ is a subset of the realizations of $r_{k}$.  Our goal is
to show that it is equal to the set of realizations of $r_{k}$, since this
would mean that there is a transition from a realization of $r_{k-1}$ to a 
realization of $r_{k}$ which can be appended to a specific realization of 
$r_{0},\dots,r_{k-1}$
(by the inductive hypothesis) to give the full desired realization of 
$r_{0},\dots,r_{k}$.

In the trivial case, where $x$ was the only element in the pile, there is
only one realization of $r_{k}$.  $S_{\text{pop}}$ must be nonempty, 
so this implies that $S_{\text{pop}}$ is equal the set of all realizations
of $r_{k}$.

Now consider the case where the pile contains at least one element besides
$x$.  We know from the inductive hypothesis that there are realizations of
$r_{0},\dots,r_{k-1}$ ending at every possible realization of $r_{k-1}$. 
Let $m_{k-1}$ denote the number of such possible realizations.

If the moving of the element $x$ to the output causes the top twinstack to be popped,
then the number of different realizations of $r_{k}$ is $\frac{m_{k-1}}{2}$.
Notice that no state in $\mathfrak{D}$ ever has more than two incoming 
edges corresponding to pop operations sending a fixed element $x$
to the output.  Therefore, the cardinality of $S_{\text{pop}}$ is greater than or equal to 
$\frac{m_{k-1}}{2}$.  So $S_{\text{pop}}$ must equal the set of all realizations
of $r_{k}$.

Alternatively, if the movement of $x$ to the output did not cause the top
twinstack to be popped, then the deque lists of every realization of 
$r_{k-1}$ must still be unique after the removal of $x$.  Therefore,
the cardinality of $S_{\text{pop}}$ must be $m_{k-1}$, implying 
that it is equal to the whole set of realizations of $r_{k}$.

Finally, in case 1 we are adding a new twinstack containing the element $x$
to the top of the pile. Consider an arbitrary $\alpha$ giving an
arbitrary realization of $r_{k}$. By the inductive hypothesis, there
is a realization of $r_{0},\ldots,r_{k-1}$ending in $\phi\left(\alpha,r_{k-1}\right)$.
But $\phi\left(\alpha,r_{k-1}\right)$ is a state of the deque which
can clearly transition to $\phi\left(\alpha,r_{k}\right)$ by taking
either an $a$ or $b$ operation. Thus there is a realization of $r_{0},\ldots,r_{k}$
ending at $d_{k}=\phi\left(\alpha,r_{k}\right)$.

Thus, if $r=r_{0},r_{2},\ldots,r_{k}$ is any run of the Rosenstiehl-Tarjan-Modified
algorithm starting with $\pi$ in the input and ending after $k$
steps at $r_{k}\in\mathfrak{R}$, for every realization of $r_{k}$,
there exists a valid run of the deque switchyard which realizes $r_{0},\ldots,r_{k}$
and ends at that particular realization of $r_{k}$. 
\end{proof}

\begin{lem}
After every run of the Normalize subroutine which returns true, every
realization of the state of the Rosenstiehl-Tarjan-Modified algorithm
is a non-sandwich state. Additionally, when the Normalize subroutine
returns false, this is because every realization of the state at the
start of that subroutine was a sandwich state.
\end{lem}

\begin{proof}
Suppose for the sake of contradiction that the Normalize subroutine
returns true, and that there is a realization which is a sandwich
state. The only way that there could be a realization as a sandwich
state is if some realization puts an element $i$ between two larger
elements, $j$ and $k$. Consider the empty deque list, to which we
begin adding elements from the stacks of our twinstacks in order to
form a realization. In forming a realization, each element must go
either to the left or right of this initial empty list. Clearly, $i$
cannot go on the opposite side from $j$ and $k$ while still appearing
between them in the realization. Thus at least one of $j,k$ must
go on the same side as $i$ and on the outside of $i$ with respect
to the position of the initial empty list.

Assume without loss of generality that the element $i$ is located
between the initial empty stack and the element $j$ in the realization
as a sandwich state. Then $j$ must have been located either above
$i$ in the same stack of the twinstack containing $i$, or $j$ must
be in a twinstack above the one containing $i$. But this is a contradiction,
since it is clear that the pile is normalized every time the Normalize
subroutine of Rosenstiehl-Tarjan-Modified returns true. Thus, Normalize
cannot return true unless every realization of the algorithm state
is a non-sandwich state.

Now we wish to show that, when the Normalize subroutine returns false,
every realization of the state at the start of that subroutine call
was a sandwich state. Suppose that Normalize returns false. Then,
at the time of the first execution of the return, the second twinstack
is double-sided, and the top elements of both of its sides are smaller
than an element $x$ in the top twinstack. The previous calls of Normalize,
if any, only modified the state by welding together elements of the
top two stacks. Thus, when Normalize was first called, this double-sided
twinstack was already double-sided, and the element $x$ was already
in a twinstack above the double-sided twinstack. Call this state $r\in\mathfrak{R}$.

By making sure that the bottom twinstack is always one-sided if it
is monotonic, we ensure that the presence of a double-sided twinstack
ensure that any realization of that twinstack and all those below
it must be non-monotonic. Thus, any realization of the state $r$
must have the top elements of the double-sided twinstack separated
by an element larger than both of them. So, wherever the realization
places $x$, it will sandwich one of these top elements of the double-sided
twinstack. Therefore, every realization of $r$ is a sandwich state.
\end{proof}

\begin{thm}
The Rosenstiehl-Tarjan-Modified algorithm is correct. That is, it
returns true for a permutation $\pi$ if and only if there is a valid
run of a deque switchyard which sorts the permutation $\pi$.
\end{thm}

\begin{proof}
($\Longleftarrow$): Suppose that the Rosenstiehl-Tarjan-Modified
algorithm returns true. Then, at the time of return, its input list
must be empty. 

We claim furthermore, that at the time of return, every element of
the permutation is in the output list in sorted order. Clearly every
element that is in the output list must be in sorted order.  
Suppose, for the sake of contradiction, there are elements remaining
in the pile at the time of return.  Since the pile was normalized 
prior to the final call to ToOuput, the smallest element not on
the output must have been one of the top elements of the top
twinstack.  But this would imply that it would have been moved 
by ToOuput.  Thus we have a contradiction, and the pile must
be empty at the time when the algorithm terminates with the return
value true. Therefore, at the time of return, every element of the
permutation is in the output list in sorted order.

At the time of return, the states taken by the algorithm represent
a run of Rosenstiehl-Tarjan-Modified staring with the state where
$\pi$ is in the input, and ending with the state where $1\ldots n$
is in the output. By the previous lemma, every such run is realizable
as a run of a deque switchyard, with realizations for each possible
realization of the final state of the Rosenstiehl-Tarjan-Modified
run. Any such realization in this case represents a successful sorting
of $\pi$. Therefore $\pi$ is deque sortable.

($\Longrightarrow$): Now suppose that $\pi$ is deque sortable
permutation.  The one remaining result we need to show that 
Rosenstiehl-Tarjan-Modified will return true is the following subclaim.

\begin{sclm}
Suppose a run of Rosenstiehl-Tarjan-Modified transitions from a state
$r\in\mathfrak{R}$ to $r'\in\mathfrak{R}$. Let $d$ be any realization
of $r$ which is not a sandwich state, and let $d'$ be any non-sandwich
state reachable by taking a valid step from $d$ as part of a successful
reduced run of the deque switchyard. Then, either $d$ is a realization
of $r'$, or $d'$ is a realization of $r'$.
\end{sclm}

\begin{proof}
Once again, we consider all possible state transitions of the Rosenstiehl-Tarjan-Modified
algorithm.
\begin{enumerate}
\item A new twinstack could be added to the top of the pile by the FromInput
subroutine.
\item An element can be moved from the pile to the output (possibly with
the removal of the containing twinstack).
\item The top twinstack could be welded down.
\item A twinstack could be reversed.
\item The single largest element of the bottom twinstack, which is currently
residing as the only element in the stack in one side of that twinstack,
can be moved to the bottom of the other stack of that twinstack. (This
can occur in the ToOuput and Normalize subroutines.)
\end{enumerate}
In the last two cases, and realization before the transition is still
realizable after the transition, so the statement holds.

In case 3, the only realizations being eliminated by the weld are
those where some large element $j$ would have gone on the outside
of a smaller top element $i$. Such realizations are always sandwich
states unless $i$ is the only element in all twinstacks below the
one containing $j$. This situation would be handled by the 5th case
instead, so we can safely conclude that any realizations eliminated
by the weld are sandwich states.

Now consider case 2. In this case we are popping an element which
can be placed on the output. According to the definition of a reduced
run, the only step we could possibly be taking from any realization
$d$ of $r$ would be to pop this element. Therefore, every $d'$
which is reachable by a valid step of a reduced run of the deque switchyard
is a realization of $r'$. 

Finally, consider case 1. Here, we are transitioning from $r$ to
$r'$ by pushing a new element from the input onto the pile as its
own new twinstack. Since the FromInput call involved in this transition
from $r$ to $r'$ was immediately preceded by a
call to ToOuput, the state $r$ cannot have any elements which are
available to be popped. Therefore, the only possible transitions from
$d$ to $d'$ as part of a successful reduced run on the deque switchyard
are caused by the operations $a$ and $b$, and the two resulting
states are both realizations of $r'$.

Therefore, in every case, if $d$ is any realization of $r$ which
is not a sandwich state, then either $d$ is a realization of $r'$,
or every $d'$ which is reachable by a valid step from $d$ as part
of a successful reduced run of the deque switchyard is realizable
from $r'$. 
\end{proof}

Since $\pi$ is deque sortable, there must be a successful run of
the deque switchyard starting with $\pi$ in the input and ending
with the identity permutation in the output. By our lemma, there must
also be a successful reduced run. Denote this successful reduced run
by $d=d_{0},\ldots,d_{2n}$.

Because the initial state $d_{0}$ is a realization of the initial
state of the run of Rosenstiehl-Tarjan-Modified on the permutation
$\pi$, the above subclaim shows that, as we run the Rosenstiehl-Tarjan-Modified
algorithm, every state $d_{0}$ through $d_{k}$ will be included
as realizations of this run, where $d_{k}$ is the state which is
a realization of the current state of the Rosenstiehl-Tarjan-Modified
at the time of termination. 

Suppose for the sake of contradiction that the Rosenstiehl-Tarjan-Modified
were to return false. Then our lemma shows that every realization
of the current state of the Rosenstiehl-Tarjan-Modified algorithm
must be a sandwich state. But $d_{k}$ is a realization which is not
a sandwich state, which gives a contradiction. Therefore the Rosenstiehl-Tarjan-Modified
algorithm must return true. 

So the Rosenstiehl-Tarjan-Modified algorithm returns true for a permutation
$\pi$ if and only if there is a valid run of a deque switchyard which
sorts the permutation $\pi$. 
\end{proof}

\section{Applying Rosenstiehl-Tarjan-Modified to the Enumeration Problem}

The naive approach to calculating the number of sortable permutations
of a given length $n$ would be to enumerate all permutations of length
$n$ and then test each one for sortability using the appropriate
version of the Rosenstiehl-Tarjan algorithm. This approach has the
ghastly runtime of $\Theta\left(n\cdot n!\right)$.

A much better approach is to search the permutation tree, pruning
subtrees of permutations which are not sortable. Consider the tree
where each node at depth $k$ is a permutation of length $k$, and
its children are the permutations of length $\left(k+1\right)$ formed
by inserting the element $\left(k+1\right)$ at the $\left(k+1\right)$
possible insertion locations. 

\begin{figure}[ht]
\center{\includegraphics[scale=0.6]{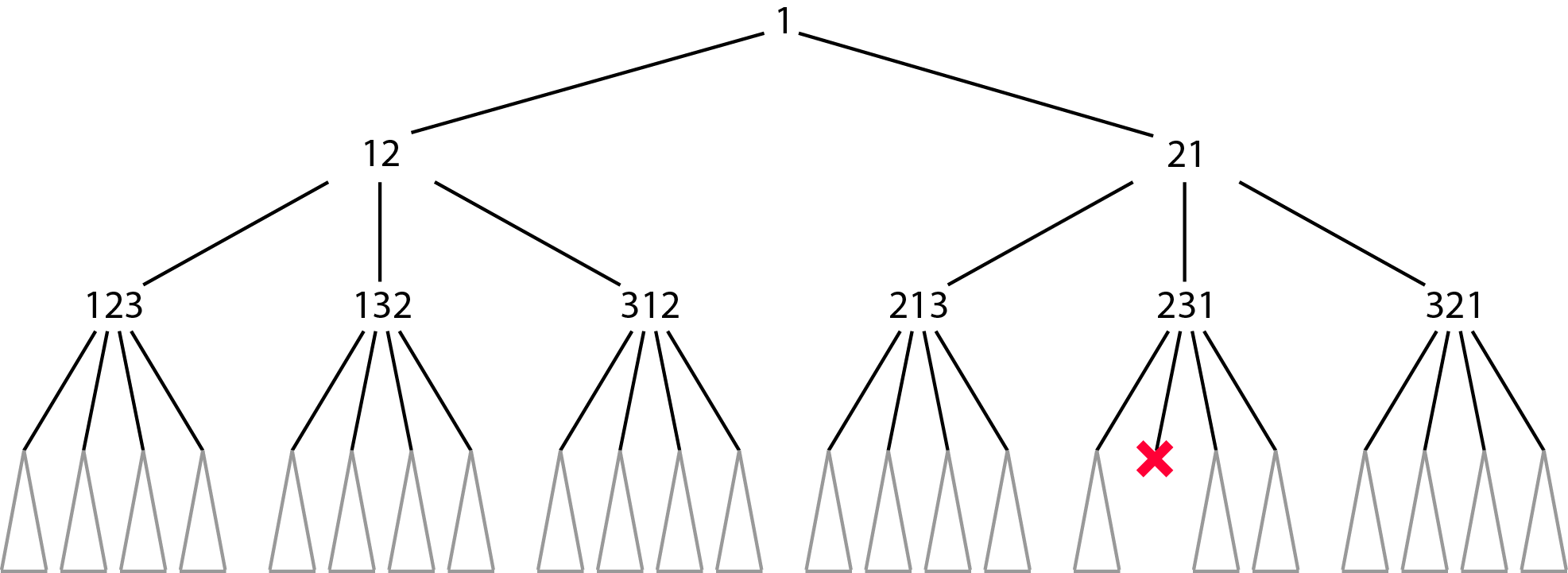}}
\caption{Searching the permutation tree}
\end{figure}

Clearly, this tree is the poset of all permutations (oriented
upside down, and with only some of the connections shown). Since permutations classes are closed under downward
containment in the poset of permutations, they are closed under upward
traversal of this tree. Therefore, any node of the tree whose permutation
does not belong to a given permutation class is the root of a subtree
which does not contain any members of that class. We use this property
to prune a depth first search of the permutation tree for sortable
permutations of length $n$.

The algorithm for calculating $\left|\mathcal{C}_{n}\right|$ or $\left|\mathcal{D}_{n}\right|$
is thus given as follows:
\begin{itemize}
\item Start at the root of the permutation tree and traverse it via depth
first search.
\item At each node, test the permutation for sortability using the appropriate
version of the Rosenstiehl-Tarjan-Modified algorithm. If the permutation
is not sortable, backtrack.
\item Whenever a permutation of length $n$ is found to be sortable, increment
the number of sortable permutations, then backtrack.
\end{itemize}
The runtime of this algorithm clearly depends on the number of nodes
visited. Since every node visited is the child of some node at the
previous depth whose permutation is sortable, and each node at depth
$\left(k-1\right)$ has $k$ children, the number of nodes visited is given by the
following formula (for the parallel stack sortability case).\begin{align*}
\text{\# nodes visited} & =\sum_{i=1}^{n}i\left|\mathcal{C}_{i-1}\right| &  & \text{where we define }\left|\mathcal{C}_{0}\right|=1\\
 & \leq n\sum_{i=1}^{n}\left(\text{gr}\left(\mathcal{C}\right)\right)^{i-1}\\
 & =n\sum_{i=0}^{n-1}\left(\text{gr}\left(\mathcal{C}\right)\right)^{i}\\
 & =n\frac{\left(\text{gr}\left(\mathcal{C}\right)\right)^{n}-1}{\text{gr}\left(\mathcal{C}\right)-1}\\
 & =O\left(n\cdot\left(\text{gr}\left(\mathcal{C}\right)\right)^{n}\right)\end{align*}

Combining this calculation with the linear runtime at each visited
node, we see that the runtime of this enumeration algorithm is $O\left(n^{2}\cdot\left(\text{gr}\left(\mathcal{C}\right)\right)^{n}\right)$.
(Similarly, for the deque sortability case we derive a runtime of
$O\left(n^{2}\cdot\left(\text{gr}\left(\mathcal{D}\right)\right)^{n}\right)$.) 

Modulo the sub-exponential factor, this is optimal among algorithms
which must consider every sortable permutation. This is asymptotically
superior to the runtime of Zimmermann's C program (which can be though
of as having a runtime of $O\left(\left(\text{gr}\left(\mathcal{D}\right)+\Delta\right)^{n}\right)$
for some positive constant such that 
$\text{gr}\left(\mathcal{D}\right)<\left(\text{gr}\left(\mathcal{D}\right)+\Delta\right)<16$),
though the runtime difference is not sufficient to change the range
of values of $n$ for which the calculation can be reasonable performed.
We wrote an efficient C implementation of this algorithm which calculated 
the first $14$ terms of the
sequences $\left|\mathcal{C}_{1}\right|,\left|\mathcal{C}_{2}\right|,\ldots$
and $\left|\mathcal{D}_{1}\right|,\left|\mathcal{D}_{2}\right|,\ldots$ (the
same terms that Zimmermann calculated with his algorithm).

We now transition to the construction of a new algorithm whose operation
does not depend on examination of each sortable permutation.

\section{The Relativistic Algorithm for Counting Parallel Stack Sortable Permutations}
As with the Rosenstiehl-Tarjan-Modified algorithm itself, we will
develop the new algorithm by first considering the simpler case of
computing $\left|\mathcal{C}_{n}\right|$, and then progressing to
a version of the algorithm which can calculate $\left|\mathcal{D}_{n}\right|$. 

Let us define a run of the Rosenstiehl-Tarjan-Modified algorithm as
before, as a walk on the state graph whose vertices are the elements
of $\mathfrak{R}$. Let us call a successful run of the algorithm
an \underbar{R-T history}. The key idea behind the new algorithm is
that, instead of counting the number of parallel stack sortable permutations
directly, we can count histories of the Rosenstiehl-Tarjan-Modified
algorithm. We demonstrate the equivalence of these two approaches
with the following lemma.

\begin{lem}
The set of histories of the Rosenstiehl-Tarjan-Modified algorithm
for parallel stacks (respectively deques) is in bijection with the
set of parallel stack sortable permutations (respectively deque sortable
permutations).
\end{lem}

\begin{proof}
($\supseteq$): Different sortable input permutations always produce
distinct runs, and since Rosenstiehl-Tarjan-Modified is correct, they
produce distinct histories. Therefore, there are at least as many
R-T histories as there are sortable permutations.

($\subseteq$): Consider any set of distinct R-T histories. Since
the R-T algorithm is deterministic, they must differ in their first
state, and thus in there input permutation. By the correctness of
Rosenstiehl-Tarjan-Modified, each of these permutations must in fact
be sortable. Therefore, there are at least as many sortable permutations
as there are R-T histories. 
\end{proof}

Counting R-T histories is still not an easy task. The state space
$\mathfrak{R}$ is still very large and complex. We (Doyle) noticed,
however, that there is a great deal of symmetry between states whose
pile of twinstacks have the same relative orders of elements. Thus,
we are going to consider an new state space designed to better take
advantage of these symmetries.

\begin{quotation}
$\overset{def}{=}$ Let the \underbar{relativistic twinstack} (or
an r-twinstack) corresponding to a given twinstack be the structure
obtained by forgetting all of the labelings of the elements and remembering
only their relative orders. (So an r-twinstack containing $k$ elements
can be represented by a binary string of length $k$, where the $i$th
smallest element is in the left stack if and only if the $i$th number
in the string is a $1$.) 
\end{quotation}

\begin{figure}[ht]
\center{\includegraphics[scale=0.8]{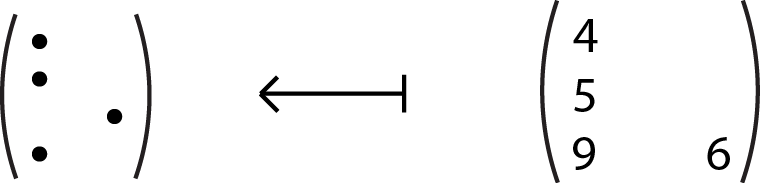}}
\caption{An example r-twinstack and one classic twinstack mapping to it}
\end{figure}

Short aside: Rosenstiehl and Tarjan require their twinstacks to be
nonempty, but we will see that it is convenient for us to also consider
``the empty r-twinstack''.

Notice that almost all r-twinstacks correspond to several different
twinstacks. Additionally, the possibilities for an r-twinstack with
$k$ elements involved in the sorting of a permutation of length $n$
are the same as the possibilities for an r-twinstack involved in the
sorting of a permutation of length $m$ as long as $k<n,m$. Thus,
by rephrasing our state space in terms of r-twinstacks, we make it
much easier to phrase the counting of histories in terms of subproblems.
This motivates the following definition.

\begin{quotation}
$\overset{def}{=}$ Let a \underbar{r-state} consist of a pile of
r-twinstacks, along with a count of the number of elements in the
input and the number of elements in the output.
\end{quotation}

There is a natural surjective mapping from states of the Rosenstiehl-Tarjan-Modified
algorithm to r-states. Let the set of all r-states be denoted $\mathfrak{E}$.
Just as we did for the states in $\mathfrak{R}$, we will consider
the simple directed graph formed by valid state transitions of the
Rosenstiehl-Tarjan-Modified algorithm, except that we consider a transition
from one r-state to another to be valid if and only if there is a
pair of states in $\mathfrak{R}$ which map to the r-states and between
which there is a valid Rosenstiehl-Tarjan-Modified state transition.
Notice that just as the correspondence between r-states in $\mathfrak{E}$
and states in $\mathfrak{R}$ can be one-to-many, the correspondance
between edges in the transition graph on vertex set $\mathfrak{E}$
and the transition graph on vertex set $\mathfrak{R}$ can also be
one-to-many. 

Furthermore, while the transition graph on $\mathfrak{R}$ had at
most one out-edge from any vertex, a single vertex in the transition
graph on $\mathfrak{E}$ can have many out-edges since there can be
many possible transitions from a given r-state depending on the relative
order of the new element taken from the input.

Example:

\begin{figure}[ht]
\center{\includegraphics[scale=1.0]{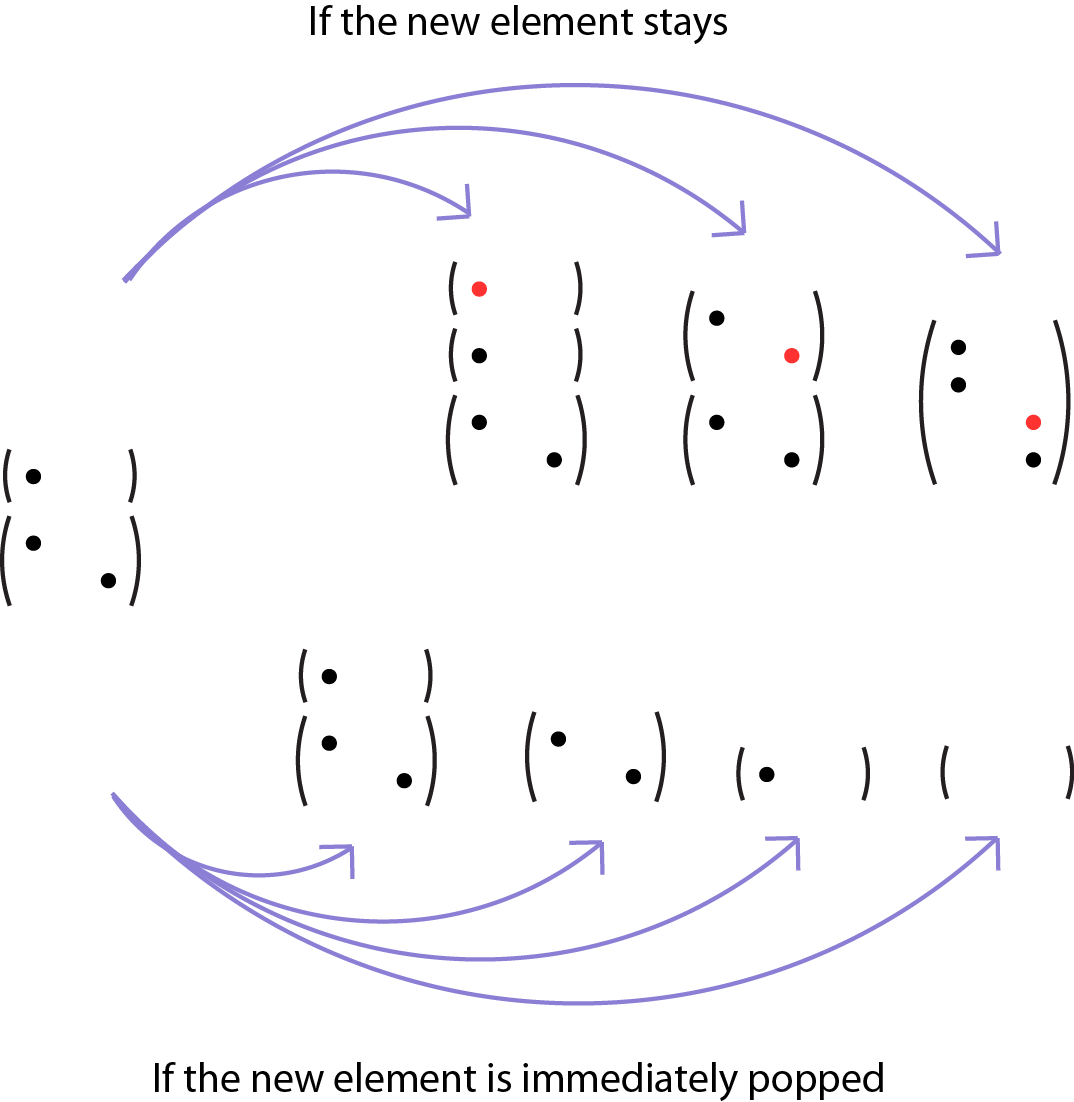}}
\caption{Example of possible transitions from a given r-state}
\end{figure}

We next go on to give the analog of our definition of an R-T history.

\begin{quotation}
$\overset{def}{=}$ An \underbar{r-history} is a length $n$ path
in the r-state graph on vertex set $\mathfrak{E}$ which has a lifting
to the state graph on vertex set $\mathfrak{R}$ as an R-T history
(or as a successful run of the Rosenstiehl-Tarjan-Modified algorithm).
\end{quotation}

The next result is slightly more surprising than the bijection between
R-T histories and sortable permutations.

\begin{lem}
The set of r-histories on the transition graph of r-states with $n$
elements is in bijection with the set of sortable permutations of
length $n$.
\end{lem}

\begin{proof}
Suppose we are given an r-history, $\alpha$. Notice that the final
r-state of $\alpha$ is the unique r-state with $n$ elements in the
output. Call this state $x$. Since every r-history lifts to an R-T
history, and every R-T history ends with the identity permutation
in the output, we can label each of the elements in the output of
the r-state $x$ so that they form the identity. Then, by following
$\alpha$ in reverse, it is possible to track the labels on the elements
as they propagate back from the output, into the pile of r-twinstacks,
and into the input. 

Thus, by running $\alpha$ in reverse, we can determine the unique
input permutation which could have lead to the that r-history. Since
the inverse of this map is the map taking a sortable permutation to
the R-T history that it generates and then mapping that R-T history
to an r-history in the natural way, this map is clearly a bijection
between r-histories and sortable permutations. 
\end{proof}

Lemma 7.2 implies that, instead of trying to count either the number
of sortable permutations or the number of Rosenstiehl-Tarjan-Modified
histories, we can simply count the number of r-histories. This will
prove much easier, since the set r-states in $\mathfrak{E}$ for permutations
of size $m<n$ are a subset of the set of states for size $n$, modulo
differences in count of the number of elements in input and output.
This is the gist of the subproblem decomposition which will be used
for the relativistic algorithm. In order to formalize this decomposition,
however, we would like to introduce a few more concepts.

We propose to let an epoch be a section of a r-history which is tied
to a certain level of the stack of r-twinstacks. The base epoch, associated
with the lowest r-twinstack, will be the r-history itself. We might
then associate one or more epochs with the second r-twinstack, and
still more with the r-twinstack above that, etc.

An epoch $E$ at a higher level can be viewed as a subpath of the
total r-history. However, we say that $E$ has its own perspective
from which it sees itself as the base epoch at the level of the bottom
twinstack. Thus $E$ views itself as the r-history created by taking
the subpath of the original r-history and removing from each of the
states in this sequence all of the r-twinstacks below $E$'s level. 

Let us now give a formal definition of an epoch. 

\begin{quotation}
$\overset{def}{=}$ Given an r-history (which is a path in the transition
graph on the set of r-states $\mathfrak{E}$), an \underbar{epoch
at level $h$} (for $h>0$) is defined as a {}``subpath'' $x'\rightarrow\ldots\rightarrow y'$
s.t.:
\begin{itemize}
\item The epoch begins at some r-state $x$ where the r-twinstack at level
$\left(h-1\right)$ is nonempty and the r-twinstack at level $h$
is empty.
\item The epoch ends at the r-state $y$ which is the first r-state following
$x$ in the path such that the r-twinstack at level $\left(h-1\right)$
is modified in the transition to r-state $y$.
\end{itemize}
And where the r-states in the {}``subpath'' $x'\rightarrow\ldots\rightarrow y'$
are created from the actual subpath $x\rightarrow\ldots\rightarrow y$
by removing all of the r-twinstacks for levels $0$ through $\left(h-1\right)$
and by decrementing the input and output element counts such that
the output count is zero at $x'$ and the input count is zero at $y'$.
(When $h=0$, the starting and ending conditions are waved, and we
say that there is a single epoch corresponding exactly to the r-history.)
\end{quotation}

Each epoch is itself an r-history for the sorting of some smaller
permutation. (Notice that the pile of r-twinstack is always empty
at the end of the epoch because, if the epoch is not the original
base epoch then it must end when the r-twinstack just below its level
is modified by either a pop or weld operation, and in either case
the r-twinstack at its level must be empty.) Furthermore, since no
epoch can end before any epochs above it end (by the requirment of
having empty r-twinstacks at the end of an epoch), the epochs are
properly nested. Every epoch at level $\left(h-1\right)$ can contain
zero or more epochs at level $h$.

To make the epoch decomposition of a given epoch/r-history well defined,
we say that an epoch begins at level $h$ \emph{any time} that the
current r-state has a nonempty r-twinstack at level $\left(h-1\right)$
and there is not already an existing epoch at level $h$. Thus every
epoch which includes an r-state with a non-empty r-twinstack must
contain \emph{one} or more child epochs at the next higher level,
while any epoch which includes only r-states with empty r-twinstacks
has no child epochs.

\begin{figure}[ht]
\center{\includegraphics[scale=1.0]{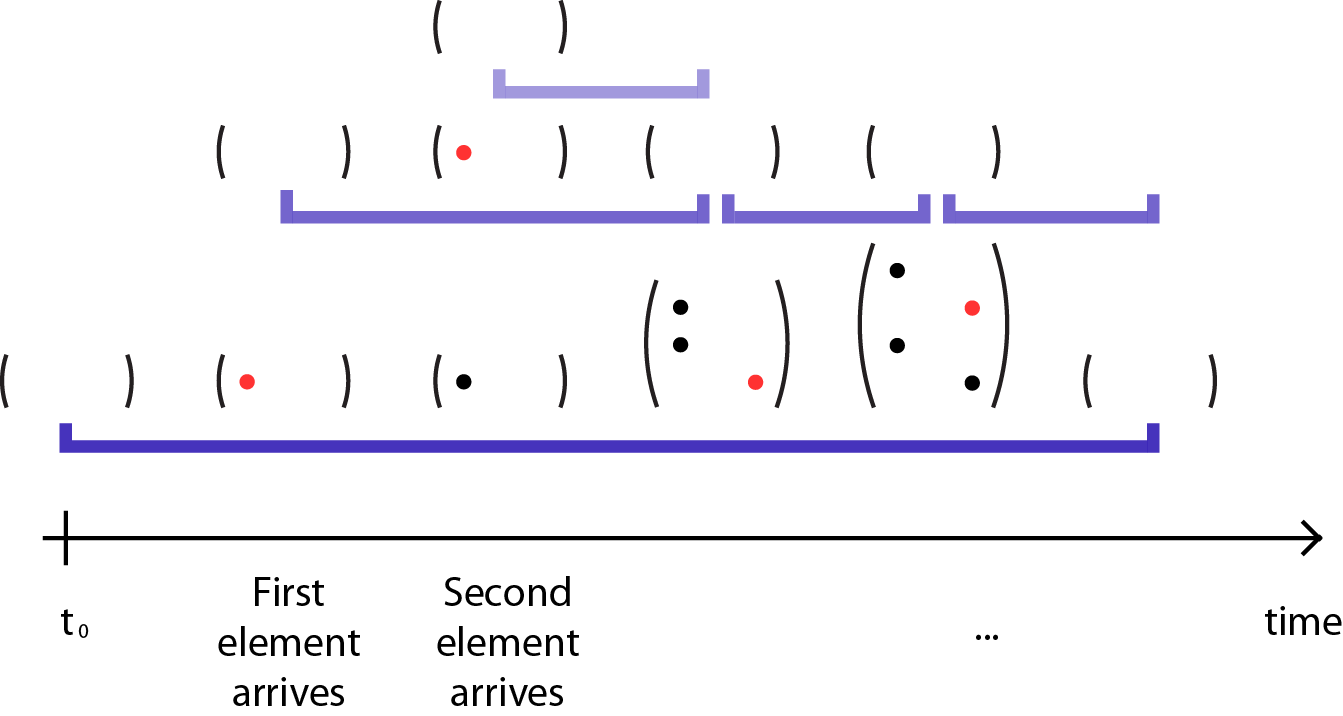}}
\caption{Example decomposition of an r-history into epochs}
\end{figure}

It should be clear that this nested tree of epochs, each corresponding
to a smaller r-history, gives a subproblem decomposition which we
can use to address the task of calculating the number of r-histories
of a given length. Before we finally address this directly, however,
allow us to introduce one more key definition relating to the ways
in which an epoch can end.

An epoch always ends either by popping all of its elements and then
popping an element of the r-twinstack below that epoch, or by welding
all of its elements onto the r-twinstack below that epoch. (The exception
is the original bottom epoch, or root epoch, which ends where it pops
all of its elements and there are no elements remaining in the input.)
Thus, the ways in which the r-twinstack belonging to a given epoch
can be modified are tied to the ways in which the epoch at the next
level above can end. 

Recall that every epoch/r-history can be lifted to some successful
R-T history. Since, whenever we weld in an R-T history there can only
be one element $i$ which is too big and is thus preventing the pile
from being normal, this is also the case in epochs/r-histories. Thus,
whenever we weld $k$ elements onto some r-twinstack $t$, we know
exactly what form they will have. There will be one large element
$i$ which is larger than some element in $t$, and on the opposite
side will be welded $k-1$ smaller elements which are smaller than
every element in $t$. Thus the integer $k\geq1$ completely describes
the possible transitions that can be caused by the weld.

\begin{quotation}
$\overset{def}{=}$ We say that at the end of every epoch $E$, $E$
sends a \underbar{signal} $k$ to the epoch below it, with $k=0$
if $E$ ends by popping, or $k$ equals the number of elements being
welded down if $E$ ends by welding.
\end{quotation}

The information sent as a signal by an epoch $E$ at its termination
is exactly what is needed to determine the ways in which the r-state
can change at the epoch below $E$ when epoch $E$ ends. We can also
think of an epoch as sending a signal even at some times when it is
not terminating, whenever it presents the opportunity for the epoch
below to modify its r-twinstack. Namely, we can view an epoch $E$
as sending $k=0$ whenever it pops every element from its r-twinstack,
and as sending the signal $k>0$ whenever it places a new element $i$
at the very bottom of its twinstack (either through welding or through
pushing from the input) where $k$ is the number of elements in its
r-twinstack.

Viewed in this way, the epoch below $E$ is presented with opportunities
to modify its twinstack. It may ignore some of these signals (leaving
its r-twinstack unchanged and the epoch $E$ unterminated). At some
point, though, the epoch bellow will accept one of these signals,
modify its r-twinstack, and end epoch $E$.

We are now ready to present the subproblem definition which we will
use for the relativistic enumeration algorithm. Let us define a new
map\begin{align*}
h\left(m,k\right)\end{align*}
to count the number of r-histories (alternatively the number of epochs)
which take $m$ steps and then end by sending signal $k$.

More generally, we want to consider starting not just from the r-state
with the empty pile of r-twinstack, but from the r-state whose pile
contains exactly one r-twinstack $S$ which may or may not be empty.
Thus we let \begin{align*}
h\left(S,m,k\right)\end{align*}
 be the map counting the number of {}``r-histories'' starting with
r-twinstack $S$, taking $m$ steps, and then ending signal $k$.
(We place r-histories in quotes because, formally, we defined r-histories
to only start with the empty r-state.)

Suppose we can give an efficient recursive calculation for $h\left(S,m,k\right)$.
(We will soon do so.) Then we will have solved
the problem of enumerating the number of parallel stack sortable permutations
of length $n$, since this is equal (by Lemma 7.2) to the number of
r-histories of length $n$ which end with all elements in the output,
and these are counted by $h\left(\left(\begin{array}{cc}
\  & \ \end{array}\right),n,0\right)$ (where $\left(\begin{array}{cc}
\  & \ \end{array}\right)$ denotes the empty twinstack).

Therefore, all that remains is to show how to compute $h\left(S,m,k\right)$
recursively. The recursion will be on the number of steps, $m$.

\subsubsection*{Recursive Case ($m>1$):}

Consider first the case where $S=\left(\begin{array}{cc}
\  & \ \end{array}\right)$. Clearly, the first step can transition to one of two different states.
Either the state whose r-twinstack contains one element, $\left(\begin{array}{cc}
\bullet & \ \end{array}\right)$, or (with an immediate pop to the output) the state whose r-twinstack is still empty, $\left(\begin{array}{cc}
\  & \ \end{array}\right)$. Therefore,\begin{align*}
h\left(\left(\begin{array}{cc}
\  & \ \end{array}\right),m,k\right) & =h\left(\left(\begin{array}{cc}
\bullet & \ \end{array}\right),m-1,k\right)+h\left(\left(\begin{array}{cc}
\  & \ \end{array}\right),m-1,k\right)\end{align*}

Now consider the case where the r-twinstack $S$ is not the empty
r-twinstack. There are two possible subcases. The first subcase is
where the r-twinstack $S$ first changes to some r-twinstack $S'\neq S$
after $i$ steps for some integer $i$ less than $m$. This change
must correspond to a signal, $j$, received from the epoch ending
when the change occurs. 

\begin{figure}[ht]
\center{\includegraphics[scale=1.0]{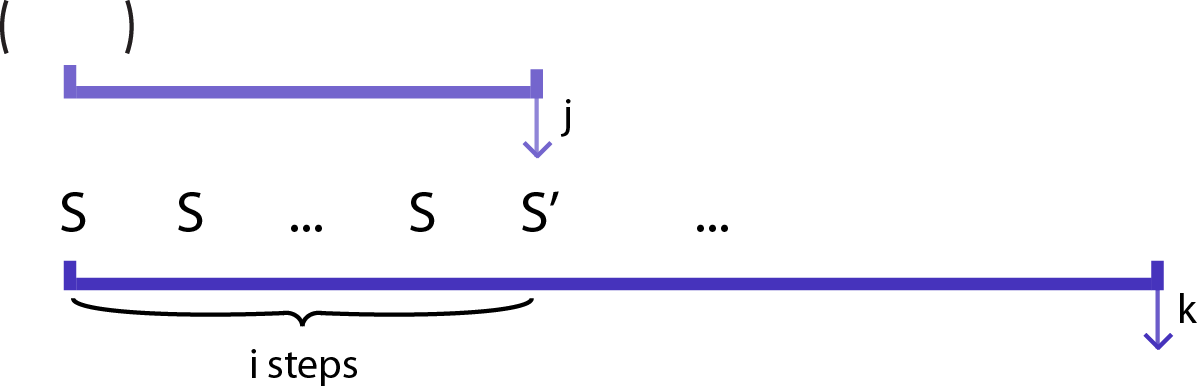}}
\caption{Example of signal passing in the case where the nonempty r-twinstack 
$S$ changes after $i$ steps for some integer $i$ less than $m$}
\end{figure}

For each selection of $1\leq i<m$, once can consider every possible
signal $j$ that could be recieved (since each such $j$ corresponds
to a distinct subset of the possible epochs at the next level beginning
at the start and lasting for $i$ steps). Then, for each possible
signal $j$, there may be many possible $S'$s that are reachable
from $S$ given that signal. Thus the number of epochs belonging to
this subcase is given by.
\begin{align*}
\sum_{i=1}^{m-1}\sum_{j\geq0}\left[h\left(\left(\begin{array}{cc}
\  & \ \end{array}\right),i,j\right)\sum_{{S'\neq S\text{ reachable}\atop {\text{from }S\text{ with}\atop \text{signal }j}}}h\left(S',m-i,k\right)\right]
\end{align*}

The second subcase is when the epoch's r-twinstack, $S$, remains
unchanged all the way till then end of the epoch. Finally, after $m$
steps, the epoch must receive a signal $k'$ from its child epoch
which causes it to send signal $k$ to its parent epoch. 

\begin{figure}[ht]
\center{\includegraphics[scale=1.0]{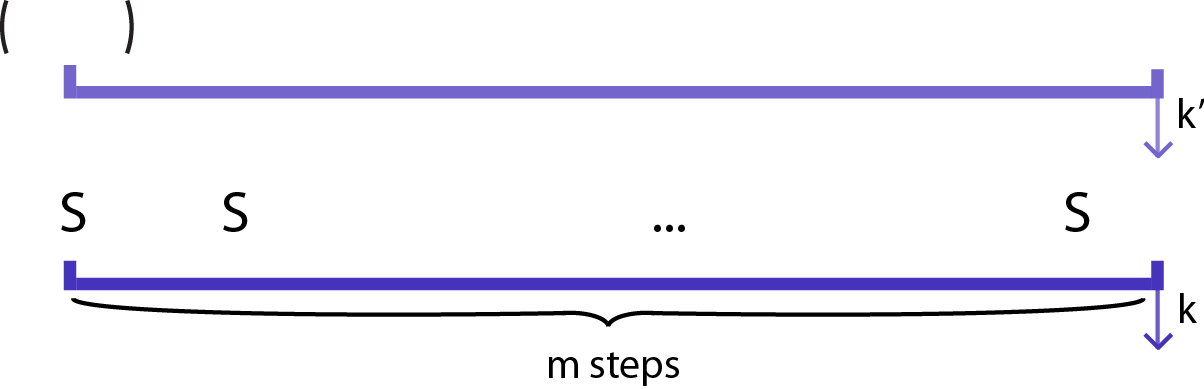}}
\caption{Example of signal passing in the case where the nonempty r-twinstack 
$S$ remains unchanged up till the end of the epoch}
\end{figure}

Because the epoch is expecting to send signal $k$ immediately upon
receiving signal $k'$, the set of signals that it can receive and
still accomplish this with is very limited. For example, if the epoch
intends to send the signal $k=0$, then this cannot be accomplished
if the received signal $k'$ is nonzero (a weld signal). However,
this can always be accomplished if $k'$ is equal to zero. Thus, for
$k=0$, $k'$ must be uniquely determined to also be zero.

When, on the other hand, the signal $k$ is greater than zero, this
indicates that the epoch ends by welding $k$ elements. A welding
end to the epoch clearly cannot occur if the signal received and accepted
from the child epoch is $k'=0$ (since that would indicate that we
pop from $S$ rather than welding its elements down). If $k'$ is
a weld signal, then the set of elements welded down from the epoch
must include every element it $S$ along with each element counted
by $k'$. Therefore, the only value which could possibly work for
$k'$ is the difference $k-\left|S\right|$ (where $\left|S\right|$
naturally denotes the number of elements in r-twinstack $S$). Notice
that whenever $S$ is one-sided, it is always possible to send signal
$k$ upon reception of signal $k'=\left(k-\left|S\right|\right)$.
Alternatively, whenever $S$ is double-sided, it is impossible to
send a weld signal, since this requires receiving a new element $i$
which is the largest in the r-twinstack $S$ which would cause the
sorting algorithm to fail (and we are only considering successful
r-histories of the sorting algorithm). Thus $k'$, if it exists, is
uniquely determined as a function of $S$ and $k$. By choosing $k'=-1$
when no $k'$ could allow us to send signal $k$, we get the following
formula for $k'\left(S,k\right)$. 
\begin{align*}
k'\left(S,k\right) & =\begin{cases}
0 & \text{if }k=0\\
\left|S\right|-k & \text{if }S\text{ is one-sided and }\left|S\right|<k\\
-1 & \text{otherwise}\end{cases}
\end{align*}

The number of epochs belonging to this second subcase is thus\begin{align*}
h\left(\left(\begin{array}{cc}
\  & \ \end{array}\right),m,k'\right) & =h\left(\left(\begin{array}{cc}
\bullet & \ \end{array}\right),m-1,k'\right)+h\left(\left(\begin{array}{cc}
\  & \ \end{array}\right),m-1,k'\right)\end{align*}

(This works with the choice of $k'=-1$ when no signal $k'$ could
enable the sending of signal $k$ because $h\left(S,m,-1\right)$
is zero for all $S$ and $m$.)

Together, the two subcases where $S$ changes to $S'$ after $i<m$
steps and where $S$ remains unchanged until after $m$ steps clearly
count all possibilities for the epochs starting with r-twinstack $S$,
taking $m$ steps, and then ending by sending signal $k$. Therefore
we get the following recursive definition of $h\left(S,m,k\right)$
for the case where $S$ is not the empty r-twinstack.\begin{align*}
h\left(S,m,k\right) & =\sum_{i=1}^{m-1}\sum_{j\geq0}\left[h\left(\left(\begin{array}{cc}
\  & \ \end{array}\right),i,j\right)\sum_{{S'\neq S\text{ reachable}\atop {\text{from }S\text{ with}\atop \text{signal }j}}}h\left(S',m-i,k\right)\right]\\
 & \qquad+h\left(\left(\begin{array}{cc}
\bullet & \ \end{array}\right),m-1,k'\right)+h\left(\left(\begin{array}{cc}
\  & \ \end{array}\right),m-1,k'\right)\end{align*}

We can arrive at a single recursive formula for $h\left(S,m,k\right)$
in both the case where $S$ is and is not the empty r-twinstack by
using the indicator function $\mathds{1}\left\{ \left|S\right|>0\right\} $
. Then \begin{align*}
h\left(S,m,k\right) & =\mathds{1}\left\{ \left|S\right|>0\right\} \sum_{i=1}^{m-1}\sum_{j\geq0}\left[h\left(\left(\begin{array}{cc}
\  & \ \end{array}\right),i,j\right)\sum_{{S'\neq S\text{ reachable}\atop {\text{from }S\text{ with}\atop \text{signal }j}}}h\left(S',m-i,k\right)\right]\\
 & \qquad+h\left(\left(\begin{array}{cc}
\bullet & \ \end{array}\right),m-1,k'\right)+h\left(\left(\begin{array}{cc}
\  & \ \end{array}\right),m-1,k'\right)\end{align*}
whenever $m$ is greater than $1$.

Since this recursive formula gives $h\left(S,m,k\right)$ in terms
of the values of the $h$ function for strictly smaller numbers of
steps, all that remains is to describe how to calculate the base case
$h\left(S,1,k\right)$ for all $S$ and $k$.

\subsubsection*{Base Case ($m=1$):}

It is immediately clear that $h\left(S,1,0\right)$ is always $1$
regardless of the value of $S$. (There is always a unique r-history
of length $1$ which ends by popping, namely the history created by
taking the next element of the input, popping it, and then popping
everything else as well.)

For the case where $k>0$ (where the epoch ends by welding, there
is always at most one r-history of length $1$ which ends by welding
$k$ elements, because any such history must take the next element
$i$ from the input and then weld it to the bottom of the r-twinstack.
These situation actually exactly parallels the second subcase of the
recursive case, in which $S$ remains unchanged until after the $m$th
step and we wanted to find a signal $k'$ whose reception would allow
the sending of signal $k$. Here, however, the signal $k'$ is limited
to being $1$. Thus we can only send signal $k$ when $S$ is one-sided
and $\left|S\right|=k-1$.

Therefore, the value of $h\left(S,1,k\right)$ is given succinctly
by\begin{align*}
h\left(S,1,k\right) & =\begin{cases}
1 & \text{if }k=0\text{ or if }S\text{ is one-sided and }\left|S\right|=k-1\\
0 & \text{otherwise}\end{cases}\end{align*}

Together with the previously derived recursive case, this gives a
complete description of $h\left(S,m,k\right)$.

We claim that this recursive description can easily be turned into
a memoized dynamic algorithm to compute $\left|\mathcal{C}_{n}\right|=h\left(\left(\begin{array}{cc}
\  & \ \end{array}\right),n,0\right)$ in $O\left(n^{5}2^{n}\right)$ time. However, since this algorithm
actually shows up as a special case of the version for the deque case,
we will wait to describe it in the next section.

\section{The General Relativistic Algorithm for Counting Sortable Permutations}
We now wish to derive a modified version of the recursive function
from the previous section which can be used to calculate the number
of deque sortable length-$n$ permutations. We begin by reviewing
the modifications that were needed to adapt the Rosenstiehl and
Tarjan's algorithm to work for deques. Recall that there were two
such modifications.
\begin{enumerate}
\item Whenever an element is welded down to become the very bottom element
of the bottom twinstack, we tuck that element under the side stack
containing any other elements (instead of leaving it on the opposite
side and thus creating a double-sided twinstack).
\item Whenever an operation popping an element element from the bottom twinstack
causes the deque to become monotonic (so that all the elements except
possibly the largest one are together in one of the side stacks),
we tuck the largest element as needed to make the twinstack single-sided.
\end{enumerate}
Clearly, these changes induce changes in the possible transitions
for an r-state. We can easily visualize the result of these changes.

\begin{figure}[ht]
\center{\includegraphics[scale=1.0]{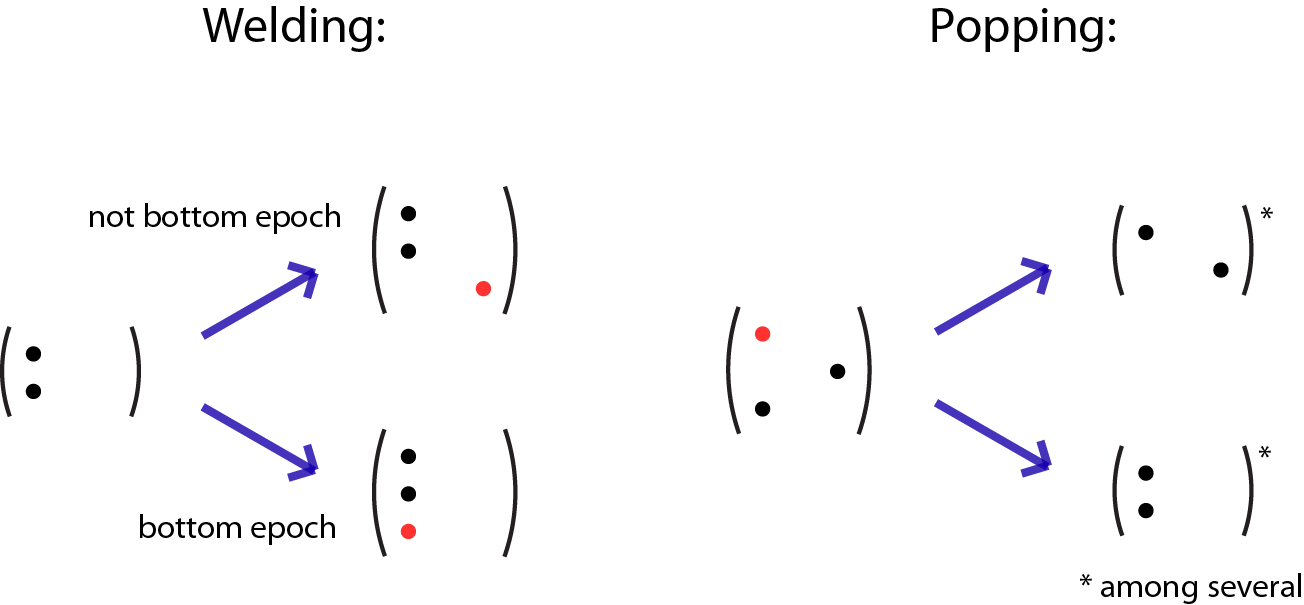}}
\caption{Transition differences}
\end{figure}

However, it is important to note that these changes only affect the
transitions that involve the actual bottom r-twinstack. Therefore,
for any higher level epochs, the number of r-histories remains completely
unchanged. Thus, our subproblem decomposition into epochs will involve
some subproblems using the new transition rules for their r-twinstacks,
as well as some subproblems using exactly the same transition rules
that we considered for the parallel stack case. This motivates the
definition of a new map $h\left(S,m,k,b\right)$ which gives the number
of epochs/r-histories starting with r-twinstack $S$, taking $m$
steps and then ending by sending signal $k$, given the transition
rules for bottom r-twinstacks if $b=1$ and given the transition rules
for non bottom twinstacks if $b=0$. Here, the new argument, $b$,
is simply a binary flag telling whether the r-histories we are counting
are at the bottom level or not. 

Note that $h\left(S,m,k,0\right)$ is just our previous map $h\left(S,m,k\right)$,
so $h\left(\left(\begin{array}{cc}
\  & \ \end{array}\right),n,0,0\right)$ still counts $\left|\mathcal{C}_{n}\right|$. When we consider $h\left(\left(\begin{array}{cc}
\  & \ \end{array}\right),n,0,1\right)$, however, we are counting the number of bottom level successful r-histories
of length $n$ using the deque transition rules. Thus $h\left(\left(\begin{array}{cc}
\  & \ \end{array}\right),n,0,1\right)$ counts $\left|\mathcal{D}_{n}\right|$, and an algorithm which can
compute $h\left(S,m,k,b\right)$ can compute the number of sortable
permutations of length-$n$ for both parallel stacks and deques. (This
was our motivation for delaying a complete algorithm description in
the previous section.)

\subsubsection*{Constructing the Recursive Formula for $h\left(S,m,k,b\right)$:}

There are two ways in which the recursive formulas for $h\left(S,m,k\right)$
needs to be modified to give recursive formulas for $h\left(S,m,k,b\right)$.
We need to correctly choose the values of $b$ to be passed to the
recursive calls, and we need to use the correct transition rules given
the passed parameter $b$.

The transition rule appears nowhere in the base case, and only in
two places in the recursive case. When we consider r-histories which
start with r-twinstack $S$ and then modify their r-twinstack to $S'$
after $i$ steps, we summed over all $S'\neq S$ such that $S'$ was
reachable from $S$ given signal $j$. For the new formula, the set
of states which are reachable depends on the value of $b$, so we
simply change the sum to be over all $S'\neq S$ such that $S'$ is
reachable from $S$ given signal $j$ and the transition rules corresponding
to $b$. 

The second place that the transition rule appears is in our statement
that the only two r-twinstack states that can be transitioned to from
the empty r-twinstack $\left(\begin{array}{cc}
\  & \ \end{array}\right)$ in one step are the empty r-twinstack, and the r-twinstack with one
element, $\left(\begin{array}{cc}
\bullet & \ \end{array}\right)$. Clearly, this is still the case regardless of whether we are using
the transition rules for the bottom epoch or not, so the addition
of $b$ does not necessitate any change to that section of the formula.

Regarding passing the correct values of $b$ to the recursive calls,
we simply need to identify which recursive calls are counting epochs
at the current level (these are passed the current value of $b$)
and which recursive calls are counting epochs at the next level up
(these are always passed $0$ as their last argument). Thus we get
the following recursive formula for $h\left(S,m,k,b\right)$ when
$m>1$.

\begin{align*}
h\left(S,m,k,b\right) & =\mathds{1}\left\{ \left|S\right|>0\right\} \sum_{i=1}^{m-1}\sum_{j\geq0}\left[h\left(\left(\begin{array}{cc}
\  & \ \end{array}\right),i,j,0\right)\sum_{{S'\neq S\text{ reachable}\atop {\text{from }S\text{ with}\atop \text{signal }j\text{ given }b}}}h\left(S',m-i,k,b\right)\right]\\
 & \qquad+h\left(\left(\begin{array}{cc}
\bullet & \ \end{array}\right),m-1,k',\mathds{1}\left\{ \left|S\right|=0\right\} \right)+h\left(\left(\begin{array}{cc}
\  & \ \end{array}\right),m-1,k',\mathds{1}\left\{ \left|S\right|=0\right\} \right)\end{align*}

The formula for the base cases of $h\left(S,m,k,b\right)$ remains
unchanged (so $h\left(S,1,k,1\right)=h\left(S,1,k,0\right)=h\left(S,1,k\right)$).

\subsubsection*{The Algorithm:}

We are now ready to describe an efficient memoized dynamic program
algorithm to compute $h\left(S,m,k\right)$. 

We will begin by describing the helper function Get-R-Twinstack-Transition-List
to efficiently compute all the $S'$s we can transition to given a
specific $S$, $j$, and $b$. Note that throughout the following
implementations it will be convenient to always place the smallest
element on the left stack.

Consider the following psuedocode for Get-R-Twinstack-Transition-List.

\begin{minipage}{\linewidth} 
\begin{algorithm}[H]
\SetAlgoRefName{Get-R-Twinstack-Transition-List$(S,j,b)$}
\DontPrintSemicolon   
\ProcSty{Procedure} \GetRTwinstackTransitionList{$S,j,b$}\;
let $result$ be a new list\;
let $l_{S}$ be the number of elements in $S$\;
let $z_{S}$ be the index of the first zero element in $S$ (equal to $l_{S}$ if $S$ contains no zeros)\;
let $m_{S}$ be the maximum size postfix of $S$ which has the property of being monotonic\;
\If{$j > 0$}{
  let $smallElements$ be a new list containing $j-1$ ones\;
  \For{$i = 1$ \KwTo $z_{S}$}{
    let $X$ be a new copy of $S$\;
    insert a $0$ into X at index $i$\;
    \If{$b == 1\  {\bf and}\ i = l_{S}$}{
      set the last element of $X$ to be a $1$\;
    }
    prepend $smallElements$ to $X$\;
    add $X$ to $result$\;
  }
}
\Else{
  \For{$i = 0$ \KwTo $l_{S}$}{
    let $X$ be a new copy of $S$\;
    remove the first $i$ elements from $X$\;
    \If{the first element of $x$ is a $0$}{
      switch the value of each element of $X$\;
    }
    \If{$b==1$\  {\bf and}\  $\left(l_{S} - i \right)\leq m_{S}$\  {\bf and}\  $X$ is not empty}{
      set the last element of $X$ to be a $1$\;
    }
    add $X$ to $result$\;
  }
  \Return{$result$}
}
\end{algorithm} 
\end{minipage}

\begin{lem}
Get-R-Twinstack-Transition-List is correct.
\end{lem}

\begin{proof}
First consider the case where $j$ is greater than zero. Here, we
are welding $j$ elements onto the r-twinstack. As discussed previously,
exactly one of these elements, $y$, must be larger than at least
one element in $S$. The other $k-1$ element must be smaller than
than every element in $S$ and on the opposite side of the incoming
weld from the element $y$. Since the $k-1$ smaller elements (if
any) will be the smallest elements in each of the resulting r-twinstacks,
we place them on the left side at the beginning of each new r-twinstate,
$X$. We then generate every every new r-twinstack $X$ over all possible
placements of the element $y$ into the right side of $X$ after the
first element of $S$ and before the first right-side element of $S$.
This process enumerates all of the possible transition states if the r-twinstack
is not the bottom stack (as indicated by $b=0$). 

Alternatively, if
$b=1$ and the last of the enumerated $X$s placed the
element $y$ as the last element, we move $y$ to the bottom of the
left stack instead (to preserve one-sidedness of the monotonic state).
Thus the returned list of $X$s is exactly the correct set of transition
r-twinstacks.

Now consider the case where $j$ is zero. Here we enumerate the results
of popping any number of the smallest elements (from none of them
to all of them). If the set of elements popped would leave the top
(smallest) element on the right stack, we reverse the resulting r-twinstack
$X$ to maintain the invariant that the smallest element is always
on the left. This will correctly give every possible transition r-twinstack
when $S$ is not the bottom r-twinstack. 

For the case when $b=1$
(indicating that $S$ is the bottom r-twinstack), we check whether each
resulting r-twinstack $X$ is monotonic. If it is, it must either
have the form of having no right stack elements, or having a single
largest element in the right stack after some number of left stack
elements. We simply change the latter case to the former by specifying
that the last element, whatever it is, must be in the left stack.

Thus, for each triple of arguments $S$, $j$, and $b$, Get-R-Twinstack-Transition-List$\left(S,j,b\right)$
returns a list with each r-twinstack $S'$ reachable from $S$ upon
reception of signal $j$ given $b$. 
\end{proof}

\begin{lem}
Get-R-Twinstack-Transition-List can be implemented to run in $O\left(\left|S\right|\right)$
time. 
\end{lem}

\begin{proof}
For every conceivably computable value of $n$, we can choose an integer
datatype having enough bits to represent each possible twinstack.
Let the last $l_{S}$ of these hold the bits indicating the position
of elements in $S$, and let all the higher order bits be zero. Then
we can perform all of the required operations of copying $S$, inserting
elements into $S$, setting elements of $S$, reversing elements of
$S$, and shifting elements of $S$, in constant time on standard
architectures using bit-shift and bitwise logic operations. 

Therefore,
the only parts of this algorithm contributing a non-constant amount
to the runtime are the computations of the variables $l_{S}$, $z_{S}$,
and $m_{S}$, and the two loops. Clearly, each of $l_{S}$, $z_{S}$,
and $m_{S}$ can be computed in $O\left(\left|S\right|\right)$ time.
Additionally, since the loops loop over $z_{S}$ and $l_{S}+1$ indices
respectively, and $z_{S}$ and $l_{S}$ are each bounded above by $\left|S\right|$,
the loops also contribute no more than $O\left(\left|S\right|\right)$
time.

Therefore Get-R-Twinstack-Transition-List can be implemented to run
in $O\left(\left|S\right|\right)$ time. 
\end{proof}

Now consider the psuedocode for the full algorithm.

\begin{minipage}{\linewidth} 
\begin{algorithm}[H]
\SetAlgoRefName{Relativistic-Histories$(S,m,k,b)$}
\DontPrintSemicolon   
\ProcSty{Procedure} \RelativisticHistories{$S,m,k,b$}\;
${\bf global}\ dictionary$\;
\If{$dictionary.hasKey((S,m,k,b))$}{
  \Return{$dictionary.getVal((S,m,k,b))$}
}
\If{$m==1$}{
  \tcc{the base case}
  \If{$k==0\ {\bf or}\ (\left|S\right|==\left(k - 1\right)\ {\bf and}\ S$ is one-sided$)$}{
    $result=1$\;
  }
  \Else{
    $result=0$\;
  }
  $dictionary.add((S,m,k,b),result)$\;
  \Return{$result$}  
}
\Else{
  \tcc{the recursive case}
  \If{$\left|S\right| == 0$}{
    $result = \RelativisticHistories{$\left(\begin{array}{cc}\bullet & \ \end{array}\right),m-1,k,b$} + \RelativisticHistories{$\left(\begin{array}{cc}\ & \ \end{array}\right),m-1,k,b$}$\;
  }
  \Else{
    $result = 0$\;
    $k' = \GetKPrime{$S,k$}$\;
    \If{{\bf not}\ $k' == -1$}{
      $result += \RelativisticHistories{$\left(\begin{array}{cc}\bullet & \ \end{array}\right),m-1,k',0$} + \RelativisticHistories{$\left(\begin{array}{cc}\ & \ \end{array}\right),m-1,k',0$}$\;
    }
    \For{$i = 1$ \KwTo $m-1$}{
      \For{$j = 0$ \KwTo $i+1$}{
        let $transitionList = \GetRTwinstackTransitionList{$S,j,b$}$\;
        \For{$S'\ {\bf in}\ transitionList$}{
          \If{{\bf not}\ $S'==S$}{
            $result += \RelativisticHistories{$\left(\begin{array}{cc}\ & \ \end{array}\right),i,j,0$}\cdot\RelativisticHistories{$S',m-i,k,b$}$\;
          }
        }
      }
    }
  }
  $dictionary.add((S,m,k,b),result)$\;
  \Return{$result$}  
}
\end{algorithm} 
\end{minipage}

Where Get-KPrime$(S,k)$ is just the previously described map 
\begin{align*}
k'\left(S,k\right) & =\begin{cases}
0 & \text{if }k=0\\
\left|S\right|-k & \text{if }S\text{ is one-sided and }\left|S\right|<k\\
-1 & \text{otherwise}\end{cases}
\end{align*}

\begin{minipage}{\linewidth} 
\begin{algorithm}[H]
\SetAlgoRefName{Relativistic-Sortable-Count$(n,b)$}
\DontPrintSemicolon   
\ProcSty{Procedure} \RelativisticSortableCount{$n,b$}\;
\If{$b==1$}{
  \tcc{the deque case}
  \Return{\RelativisticHistories{$\left(\begin{array}{cc}\ & \ \end{array}\right),n,0,1$}}
}
\Else{
  \tcc{the parallel stack case}
  \Return{\RelativisticHistories{$\left(\begin{array}{cc}\ & \ \end{array}\right),n,0,0$}}
}
\end{algorithm} 
\end{minipage}

\begin{thm}
Relativistic-Sortable-Count is correct. That is, when called with $b=1$,
it returns $\left|\mathcal{D}_{n}\right|$, and when called with with
$b=0$ it returns $\left|\mathcal{C}_{n}\right|$.
\end{thm}

\begin{proof}
Rather than offer a lengthy proof here, we simply refer the reader
to the reasoning in the above sections to see that the correct recursive
formula for $h\left(S,m,k,b\right)$ is indeed 

\begin{align*}
h\left(S,m,k,b\right) & =\mathds{1}\left\{ \left|S\right|>0\right\} \sum_{i=1}^{m-1}\sum_{j\geq0}\left[h\left(\left(\begin{array}{cc}
\  & \ \end{array}\right),i,j,0\right)\sum_{{S'\neq S\text{ reachable}\atop {\text{from }S\text{ with}\atop \text{signal }j\text{ given }b}}}h\left(S',m-i,k,b\right)\right]\\
 & \qquad+h\left(\left(\begin{array}{cc}
\bullet & \ \end{array}\right),m-1,k',\mathds{1}\left\{ \left|S\right|=0\right\} \right)+h\left(\left(\begin{array}{cc}
\  & \ \end{array}\right),m-1,k',\mathds{1}\left\{ \left|S\right|=0\right\} \right)\end{align*}
when $m>1$ and 

\begin{align*}
h\left(S,1,k,b\right) & =\begin{cases}
1 & \text{if }k=0\text{ or if }S\text{ is one-sided and }\left|S\right|=k-1\\
0 & \text{otherwise}\end{cases}\end{align*}
when $m=1$. Clearly, Relativistic-Histories implements this recursive
formula to return $h\left(S,m,k,b\right)$. 

But we also know from Lemma 7.2 that the number of length-$n$ permutations
sortable on a deque is equal to the number of root r-histories (or
epochs) of length $n$. Since a root r-history can only end by sending
signal $0$ (since there is no r-twinstack below it which it can weld
to), this latter quantity is exactly $h\left(\left(\begin{array}{cc}
\  & \ \end{array}\right),n,0,1\right)$, which is the value returned by Relativistic-Sortable-Count when
$b=1$. Thus Relativistic-Sortable-Count$\left(n,1\right)$ correctly
returns $\left|\mathcal{D}_{n}\right|$. 

Similarly, when Lemma 7.2 also implies that the number of length-$n$
permutations sortable on a pair of parallel stacks is equal to the
number of root r-histories of length $n$ using the transition rules
for parallel stacks. Once again, a root r-history can only end by
sending signal zero. Thus this count is exactly $h\left(\left(\begin{array}{cc}
\  & \ \end{array}\right),n,0,0\right)$, which is the value returned by Relativistic-Sortable-Count when
$b=0$. Thus Relativistic-Sortable-Count$\left(n,0\right)$ correctly
returns $\left|\mathcal{C}_{n}\right|$.
\end{proof}

\begin{thm}
Relativistic-Sortable-Count has time complexity $O\left(n^{5}2^{n}\right)$
and space complexity $O\left(n^{2}2^{n}\right)$.
\end{thm}

\begin{proof}
First consider the runtime of Relativistic-Sortable-Count. The Relativistic-Sortable-Count
subroutine clearly only contributes constant runtime, so any non-constant
factors must come from calls to Relativistic-Histories. Thus we need
to determine the max runtime of any given call to Relativistic-Histories,
along with the number of such calls that are being made. If Relativistic-Histories
finds that the desired value has already been computed and stored
in the memoization dictionary, then its runtime is (ostensibly) constant.
Otherwise, if $m=1$ then its runtime is still constant. Finally,
if it needs to compute the value for $m>1$, then it does so in a
nested for loop where the two outer loops can iterate order $m=O\left(n\right)$
times. Inside these two loops is the call to Get-R-Twinstack-Transition-List
(with an associated runtime of $O\left(n\right)$) and the third loop
which iterates over the return list (again $O\left(n\right)$). Thus,
in the worst case, Relativistic-Histories histories takes $O\left(n^{3}\right)$time
to run.

When we consider the number of calls made to Relativistic-Histories,
we only need to consider calls made where the desired value has not
yet been computed. (This is because whenever we have memoized the
value for a certain set of args, the runtime of Relativistic-Histories
is constant and is therefore taken care of by the computation for
the runtime of the caller.) The number of such calls to Relativistic-Histories
is limited by the size of the domain of the function $h\left(S,m,k,b\right)$.
Since $S$ can range over binary strings of length $n$, $m$ and
$k$ are both order $n$, and $b$ has only two values. The size of
this domain is $O\left(n^{2}2^{n}\right)$. Therefore, the total runtime
of all calls to Relativistic-Sortable-Count has time complexity $O\left(n^{5}2^{n}\right)$.

The space complexity of Relativistic-Sortable-Count is just the size
of the memoization dictionary, which is limited by the size of the
domain of the map $h\left(S,m,k,b\right)$. Therefore Relativistic-Sortable-Count
has space complexity $O\left(n^{2}2^{n}\right)$. 
\end{proof}

\section{Results Obtained with the Relativistic Algorithm}
As described in section 6, our most efficient implementation of the
old approach for computing $\left|\mathcal{D}_{n}\right|$ and $\left|\mathcal{C}_{n}\right|$
using tree search was only successful for up to $n=14$. In fact,
that implementation was written in C with careful consideration to
factors like avoiding memory allocation, and it still had to be run
overnight in order to compute the results for $n=14$.

Our first (and so far our only) implementation of the relativistic
algorithm is in python using the built in types, with no special emphasis
on efficiency. Such an implementation can be many orders of magnitude
slower than a good C implementation. (For example, we took same approach
of first implementing the tree search algorithm in Python before coding
it in C, and that implementation was limited to $n=10$.) Nevertheless,
because of the greatly improved asymptotic runtime our new Relativistic-Sortable-Count
algorithm, we were able to compute all of the values of $\left|\mathcal{D}_{n}\right|$
for up to $n=21$ in under twelve minutes. Similarly, we computed
$\left|\mathcal{C}_{n}\right|$ for up to $n=22$ in under twenty-two
minutes. We have included these table of numbers as appendix A and
B respectively.

While the new relativistic algorithm is much faster than the previous
best approach, this speed does come with a price. Relativistic-Sortable-Count
has a space complexity of $O\left(n^{2}2^{n}\right)$, as compared
to the linear space complexity of the tree search algorithm. (Note
that this is still an improvement over the space complexity of Zimmermann's
algorithm which had an exponential term whose base was greater than
the growth rate of the permutation class.) Because of this, the algorithm
failed to compute $\left|\mathcal{D}_{22}\right|$ or $\left|\mathcal{C}_{23}\right|$
on the linux machines on which I was running it, presumably because
the python dictionary tried to grow to well over fifteen million
elements which lead to thrashing. 

The simplest possible approaches to this problem would a more efficient
custom hash storage solution, or even just running the algorithm on
a machine with more memory. These could probably be used to acquire
a few more terms of the sequence. A better long term approach (suggested
by Peter Doyle) would be to develop a better understanding of the
dependencies among the values of the map $h\left(S,m,k,b\right)$.
This could be used to try to redesign the algorithm to use an access
order that is less affected by paging to disk. We leave such changes
for future work.

We should note before moving on, however, that fifteen million is much less than $\left(22^{2}\cdot2^{22}\right)$.
Thus the domain of the map is only being sparsely populated, and the $O\left(n^{2}2^{n}\right)$ space complexity (and $O\left(n^{5}2^{n}\right)$ time complexity) limit may 
be quite conservative.

\section{Some Observations On Deque-Sortability Given Imperfect Information}
The problem that initially caused us to start looking at the class
of permutations which were sortable on a deque ($\mathcal{D}$), was
Peter Doyle's proposal of a game he called Double-Ended Knuth (or
DEK for short). To borrow Doyle's description:

\begin{quotation}
DEK is a bare-bones relative of familiar solitaire games
like Klondike. In DEK, we use a one-suit deck consisting of only the
thirteen
hearts (say). We shuffle the deck thoroughly, and place the deck
face down
on the table. The goal is to end with the cards in a pile face up,
running in
order from ace to king. In addition to the deck and the pile (initially
empty),
we maintain a line of cards (initially empty), called the deque, spread
out
face up on the board. At any point, if the next card needed for the
pile is
available as the top card of the deck or at either end of the deque,
we may
move it up to the pile; otherwise, our only option is to move the
top card of
the deck to either end of the deque. \cite{doyle}
\end{quotation} 

Thus, DEK is the problem of sorting a permutation of length $13$
on a deque given imperfect information. If the cards forming the input
permutation were visible face up, then one could simply run our corrected
version of the Rosenstiehl-Tarjan algorithm to determine whether or
not it was sortable and if so how to sort it. Instead, however, we
are forced to make he decision of which end of the deque to add a
the top card of the deck to immediately after having revealed that
card and before viewing any of the other cards remaining in the deck.

The vast majority of our work has been spend investigating the omniscient
case.  Nevertheless, we offer a few remarks about this problem.

\begin{thm}
The distinction between sorting with complete information and sorting
with incomplete information is important. That is, one cannot choose
a strategy for the incomplete information case which will succeed
on all sortable inputs.
\end{thm}

\begin{proof}
Consider the pair of permutations $\pi=7526431$ and $\sigma=7524163$.
After revealing the first three elements of these permutations and
adding them to the deque, there are (up to reflection) two possible
states for the deque, namely\begin{align*}
\text{a) } & \mbox{257} &  & \text{and} & \text{b) } & 572\end{align*}

Clearly, state a) can be used to sort permutation $\pi$ (by adding
the $6$, $4$, $3$ and then the $1$ to the right side of the deque
and then popping everything). If however, the remainder of the permutation
happens to be $\sigma$, then the sorting attempt will fail since
the $4$ will be forced to be placed to the right of the $7$ and
then the sequence $574$ will still be on the deque when the $6$
must be placed.

Alternatively, state b) can be used to sort permutation $\sigma$.
This can be done by adding the $4$ to the left end next to the $5$,
and then adding the $1$ and sending both the $1$ and the 2 to the
output. The state of the deque is then $457$, and the $6$ and then
the three can be added to the right side before popping all of the
elements to the output. The state b) fails, however, to sort $\pi$
since the very next element, the $6$, cannot be placed without sandwiching
either the $5$ or the $2$.

Therefore, even though any permutation in the set $\left\{ \pi,\sigma\right\} $
could be sorted on a deque given complete information, it is possible
that in trying to sort a permutation from this set with incomplete
information we could fail because we are forced to make a choice about
the placement of the third element, and either choice will preclude
the possibility of sorting one of the permutations in this set. 
\end{proof}

The problem when sorting given incomplete information, as illustrated
in the above theorem, is that we must sometimes make a choice between
either of two possibilities for the placement of an incoming element
such that either choice will rule out the possibility of sorting some
subset of the permutations in $\mathcal{D}$. One wonders, therefore,
what are the necessary conditions for a choice that can affect the
sorting success.

\begin{thm}
In order for the player of a game of DEK to come across a choice
which could affect their scoring success, the following conditions
are necessary and sufficient.
\begin{enumerate}
\item The deque must already contain two distinct elements.  (Let
$i$ denote the smaller of these, and let $j$ denote the larger.)
\item The incoming element must be smaller than both end elements of the
deque.
\item There must be a gap of at least two elements between the value of
the incoming element and the smaller of the two end elements of the
deque.
\item The incoming element must not be the next element required by the output.
\item There must be an element larger than $i$ which is still in the input.
\item If the deques state is non-monotonic, then there must be an element
larger than $i$ but smaller than $j$ which is still in the input.
\end{enumerate}
\end{thm}

\begin{proof}(Necessity):
Clearly the deque state is equivalent (up to symmetry) regardless
of the player choice if the deque contains one or fewer elements.
Therefore, condition 1 is a necessary condition for being faced with
a substantive choice. 

Now suppose that the incoming element is larger than the smaller
of the two distinct end element of the deque. Clearly this incoming
element cannot be placed next to the smaller of the existing end elements,
or that smaller end element would become sandwiched and the game would
certainly be lost. Therefore, if there are two distinct end elements
on the deque and the incoming element is larger than either of them,
the players move is forced and no substantive choice exists.

Next, for condition 3, suppose that the smaller of the existing end elements of
the deque is $i$ and that the incoming element is $x=\left(i-1\right)$
or $x=\left(i-2\right)$ (so that no two element gap exists between
them). Then we claim that it is always a safe play to place the new
element next to $i$. 

Suppose that the permutation is sortable by placing $x$ next to the
the other, larger end element, $j$. Then, by placing $x$ next to
$j$ to get the state $iCjx$ (we assume wlog that $i$ is the left
end element), and then choosing future choices correctly, one will
eventually arrive at a point at which $x$ can be moved to the output.
Consider the state immediately after this move. No element can be to
the right of $j$ on the deque. No element less than $x$ can be on
the deque. No element greater than $i$ can be on the deque, since
the placement of such an element prior to the removal of $x$ would
have pinned either $i$ or $x$. Therefore, the state of the deque
must be either \begin{align*}
 & iCj &  & \text{or} &  & \left(i-1\right)iCj\quad\text{possible in the case where }x=\left(i-2\right)\end{align*}

In the former case, the only elements appearing after the point where $x$ appeared
and before the point where $x$ was popped were elements strictly less
than $x$. Therefore we could just as easily have placed $x$ next
to the end element $i$.

In the latter case, at the point when $\left(i-1\right)$ arrived,
the only other elements which had already arrived but had not already
moved to the output must lie in a sequence $S$ such that the deque
state at the moment of $\left(i-1\right)$'s arrival was $iCjxS$,
where $S$ is a sequence which is decreasing from left to right. None
of the elements arriving between $x$ and $\left(i-1\right)$ were
larger than $x$, however, so by placing $x$ next to the end element
$i$ and then inverting the placements of every element following
$x$ and preceding $\left(i-1\right)$, we could have the state $SxiCj$
at the time of $\left(i-1\right)$'s arrival. By then placing $\left(i-1\right)$
on the right end, and continuing to invert the placement of every
element received between $\left(i-1\right)$ and the movement of $x$
to the output pile, we see that it must also be safe to place $x$
next to $i$.  Thus no substantive choice is required if condition 3) does
not hold

Clearly, if $x$ is the next element required by the output then
doesn't matter where we place it since we can immediately get rid
of it. 

For condition 5, suppose that every element larger than $i$ has already
been moved out of the input. Then all such elements must already be
in the sequence $Cj$, and since we are assuming that we don't start
with a sandwich state (in which case clearly no substantive choice
can exist) the elements $n$ through $i$ must be ordered such that
they are in decending order starting from the element $n$ and reading
either right or left.

Therefore, the elements remaining in the input for any sortable permutation
must all be moved to the ends of the deque and thence to the output
before $i$ or any other element of $iCj$ is moved. This implies
that whatever remains on the input is a parallel stack sortable permutation
which can be sorted on the two parallel stacks radiating to the
left and to the right of $iCj$, and clearly the choice of which of the two
parallel stacks to add the first element to is arbitrary. 

Finally suppose that the deque is non-monotonic and that no element
between the values of $i$ and $j$ is on the input at the time that
$x$ arrives. Then for any sortable permutation, every element of
the input which is larger than $i$ must wait to arrive until after
$i$ has been moved to the output. Thus every such element must follow
every element smaller than $i$ in the input. Therefore, the input
has as a prefix some parallel stack sortable permutation consisting
of all elements less than $i$ and not yet in the output, and so once
again the placement of the element $x<i$ is unimportant.

Thus all six conditions are necessary for the player to be presented
with a substantive choice.

(Sufficiency): 
Suppose that all six conditions are met. 

\begin{sclm}
There is a permutation that is sortable only if $x$ is placed next
to the smaller end element, $i$.
\end{sclm}

\begin{proof}
Suppose first that the deque is monotonic. Let $z$ be some element
which is larger than $i$ and is still in the input. Then the permutation
in which the input consists of $xz$ followed by every remaining element
in sorted order is clearly sortable if $x$ is placed next to $i$
but not if it is placed next to $j$. 

Now suppose that the deque is non-monotonic. Then condition 6 guarantees
that there is some element $z$ which is in the input and has value
between $i$ and $j$. The same permutation is thus sortable if $x$
is placed next to $i$ but not if it is placed next to $j$.
\end{proof}

\begin{sclm}
There is a permutation that is sortable only if $y$ is placed next
to the larger end element, $j$.
\end{sclm}

\begin{proof}
Suppose first that the deque is monotonic. Consider the permutation
where $x$ is followed by $\left(i-1\right)$, then by every element
less than $x$ not already in the output in sorted order, then by
some $z>i$, and then by every remaining input element in sorted order.
If $x$ is placed next to $j$, then $(i-1)$ can be placed next to
$i$, and it is clear that the remainder of the permutation can be
successfully sorted from this state. If, however, $x$ is placed next
to $i$, then the $\left(i-1\right)$ must be placed adjacent to $j$.
After the sequence of elements less than $x$ arrives and departs,
the deque will still have a non-monotonic state with end elements
$i$ and $\left(i-1\right)$. Thus the element $z$ will necessarily
cause a sandwich.

Now suppose that the deque is non-monotonic. Let $z$ be the element
whose existence is guaranteed by condition 6. Consider the same input
permutation described above. Clearly this is still sortable given
the placement of $x$ adjacent to $j$. Once again, however, if $x$
is placed adjacent to $i$, the element $\left(i-1\right)$ must go
adjacent to $j$, and we end up with a non-monotonic stack with end
elements $i$ and $\left(i-1\right)$when $z$ arrives. 
\end{proof}

By the above two subclaims, whenever all six conditions are met, the
choice presented to the player is substantive.  Therefore these six conditions
are both necessary and sufficient.
\end{proof}

Having determined when choices matter, we want
to understand how to make the right choice. One obvious strategy for cases where we have sufficient
computational power is:

\begin{strat1}
Enumerate all possible remaining inputs, and make the
choice that leaves more of these winnable. 
\end{strat1}

After identifying this strategy, however, we realized that it amounts
to choosing based on which placement gives the player the most winnable
scenarios given omniscient information in the future. Thus this is
the optimal strategy for a modified version of DEK where the player
plays till their first substantive choice, makes that choice, and
then reveals the remainder of the input deck and trys to play on with
complete information. 

The actual optimal strategy of DEK play is this one.

\begin{strat2}
(optimal) Use a choice criteria $C$ such which will lead
to the most winnable scenarios when applied to this and all future
choices.
\end{strat2}

A priori, it seems possible that the scenarios which are winnable
from one choice are more or less evenly split beneath a future choice,
whereas the scenarios winnable from the alternative choice are not
so limited by future choices. Thus we might imagine that these two
strategies could disagree. In order to try to find an example where
the disagreed, I wrote a persistent version of Rosenstiehl-Tarjan-Modified
and then used this to calculate the decision of each strategy at each
substantive choice encountered in a search of the permutation tree.

Surprisingly, for the small cases I tested (up to $n=12$), we did
not find any example where the selections made by Strategies 1 and
2 differ.

\section{Conclusions and Acknowledgements}
To sum up the main results of this work:  We examined the deque sortability testing algorithm presented thirty 
years ago by Rosenstiehl and Tarjan, and identified an error in this algorithm.  To the best 
of our knowledge, this flaw was previously unknown.  (We have examined works which 
cite Rosenstiehl and Tarjan's algorithm, and none of them address this issue.)
Sadly, we have been unsuccessful in our attempts to contact the authors directly.
After identifying the flaw in the Rosenstiehl and Tarjan's algorithm, we proposed a
solution and then offered a proof that the modified version of the algorithm is indeed
correct.

We then developed a new algorithm for computing the number of permutations of
a given size $n$ which are sortable on either a pair of parallel stack or on a deque,
which has a greatly improved asymptotic runtime when compared with the previous 
best approach to making these calculations.  Using our new algorithm we calculated
the number of sortable permutations for several lengths beyond what was previously
known.

Finally, we have presented a description of exactly when, in attempting to sort
a permutation given incomplete information, one must make a choice which 
effects the set of permutations for which the sorting computation being attempted
can succeed.

I would like to thank both of my advisors on this project: Scot Drysdale, who has always
been very supportive has given excellent feedback in putting together this project, 
and Peter Doyle, without whom this work and my time at Dartmouth in general would
have been greatly impoverished.

I would also like to thank Professor Prasad Jayanti for kindly serving on my thesis 
committee.

Finally, I want to thank Sergi Elizalde who taught my combinatorics courses addressing
permutations and permutation patterns, and who provided great expertise that I should 
have taken advantage of sooner.

\section*{Appendix A}
This table lists our computed values of $\left|\mathcal{D}_{n}\right|$
for $n=1,\ldots,21$. The first fourteen terms of this sequence appear in the online encyclopedia
of integer sequences as sequence A182216.\begin{align*}
 & 1 &  & 51069582\\
 & 2 &  & 365879686\\
 & 6 &  & 2654987356\\
 & 24 &  & 19473381290\\
 & 116 &  & 144138193538\\
 & 634 &  & 1075285161294\\
 & 3762 &  & 8076634643892\\
 & 23638 &  & 61028985689976\\
 & 154816 &  & 463596673890280\\
 & 1046010 &  & 3538275218777642\\
 & 7239440\end{align*}

\section*{Appendix B}
This table lists our computed values of $\left|\mathcal{C}_{n}\right|$
for $n=1,\ldots,22$. \begin{align*}
 & 1 &  & 24180340\\
 & 2 &  & 161082639\\
 & 6 &  & 1091681427\\
 & 23 &  & 7508269793\\
 & 103 &  & 52302594344\\
 & 513 &  & 368422746908\\
 & 2760 &  & 2620789110712\\
 & 15741 &  & 18806093326963\\
 & 93944 &  & 136000505625886\\
 & 581303 &  & 990406677136685\\
 & 3704045 &  & 7258100272108212\end{align*}

\bibliographystyle{plain}
\bibliography{thesis_references}

\end{document}